\documentclass[11pt]{amsart}%
\usepackage{amscd,amssymb,verbatim}
\usepackage{amsmath}
\usepackage{amsfonts}
\usepackage{amssymb}
\usepackage{graphicx}%
\providecommand{\U}[1]{\protect\rule{.1in}{.1in}}
\newtheorem{thm}{Theorem}[section]
\newtheorem{cor}[thm]{Corollary}
\newtheorem{lem}[thm]{Lemma}
\newtheorem{prop}[thm]{Proposition}
\newtheorem{Def}[thm]{Definition}

\numberwithin{equation}{section}

\begin{document}
\title{On ordinal ranks of Baire class functions}

\begin{abstract}
The theory of ordinal ranks on Baire class 1 functions developed by Kechris
and Loveau was recently extended by Elekes, Kiss and Vidny\'{a}nszky to Baire
class $\xi$ functions for any countable ordinal $\xi\geq1$. In this paper, we
answer two of the questions raised by them in \cite{EKV}. Specifically, we
show that for any countable ordinal $\xi\geq1,$ the ranks $\beta_{\xi}^{\ast}$
and $\gamma_{\xi}^{\ast}$ are essentially equivalent, and that neither of them
is essentially multiplicative. Since the rank $\beta$ is not essentially
multiplicative, we investigate further the behavior of this rank with respect
to products. We characterize the functions $f$ so that $\beta(fg)\leq
\omega^{\xi}$ whenever $\beta(g)\leq\omega^{\xi}$ for any countable ordinal
$\xi.$

\end{abstract}
\author{Denny H.\ Leung}
\address{Department of Mathematics, National University of Singapore, Singapore 119076}
\email{matlhh@nus.edu.sg}
\author{Hong-Wai Ng}
\address{School of Physical and Mathematical Sciences, Division of Mathematical
Sciences, Nanyang Technological University, Singapore 637371 }
\email{hongwaing@ntu.edu.sg}
\author{Wee-Kee Tang}
\address{School of Physical and Mathematical Sciences, Division of Mathematical
Sciences, Nanyang Technological University, Singapore 637371 }
\email{WeeKeeTang@ntu.edu.sg }
\thanks{Research of the second and third author are partially supported by AcRF
project no.\ RG26/14.}
\maketitle

\section{Introduction}

Let $X$ be a Polish space; that is, a separable completely metrizable
topological space. A continuous real valued function $f$ on $X$ is said to be
of \emph{Baire class $0$}. Denote the class of Baire class $0$ functions by
${\mathfrak{B}}_{0}(X)$. Inductively, suppose that $\xi$ is a nonzero
countable ordinal and the space of Baire class $\zeta$ functions,
${\mathfrak{B}}_{\zeta}(X)$, has been defined for all $\zeta<\xi$. A real
valued function $f$ on $X$ is said to be of Baire class $\xi$ if it is a
pointwise limit of a sequence of functions in $\cup_{\zeta<\xi}{\mathfrak{B}%
}_{\zeta}(X)$. In \cite{K-L}, Kechris and Louveau developed the theory of
several ordinal ranks on Baire class $1$ functions. These ranks have their
precedents in the work of Bourgain \cite{B}, Gillespie and Hurwicz \cite{GH},
Haydon, Odell and Rosenthal \cite{H-O-R}, and Zalcwasser \cite{Z}, among
others. The ranks provide a finer measure of the complexity of Baire class $1$
functions and ramify ${\mathfrak{B}}_{1}(X)$ into an increasing transfinite
sequence of subspaces ${\mathfrak{B}}_{1}^{\xi}(X)$ (the small Baire classes).
Recently, Elekes, Kiss and Vidny\'{a}nszky \cite{EKV} extended the theory of
ordinal ranks to Baire class $\xi$ functions for any countable ordinal $\xi$.
Their work leaves a number of open questions, several of them concerning the
behavior of the ranks on unbounded Baire class functions. The main purpose of
this paper is to answer two of the questions raised in \cite{EKV}.
Specifically, it is shown that the ranks $\beta_{\xi}^{\ast}$ and $\gamma
_{\xi}^{\ast}$ are essentially equivalent, and that neither of them is
essentially multiplicative (see definitions below). Since the rank $\beta$ is
not essentially multiplicative, it is worthwhile to investigate further the
behavior of this rank with respect to products. In the last section, we
consider the multiplier problem for the oscillation rank $\beta$ on
${\mathfrak{B}}_{1}(X)$. That is, for any countable ordinal $\xi$, we
characterize the functions $f$ so that $\beta(fg)\leq\omega^{\xi}$ whenever
$\beta(g)\leq\omega^{\xi}$.

Let us recall the definitions of the ranks that will be the chief concern of
this paper.

Let ${\mathcal{C}}$ denote the collection of all closed subsets of $X$. A
\emph{derivation} on ${\mathcal{C}}$ is a map $\delta:{\mathcal{C}}%
\rightarrow{\mathcal{C}}$ such that $\delta(P)\subseteq P$ for all
$P\in{\mathcal{C}}$. Iterate $\delta$ in the usual way: $\delta^{0}(P)=P$,
$\delta^{\alpha+1}(P)=\delta(\delta^{\alpha}(P))$ for any countable ordinal
$\alpha$ and $\delta^{\alpha}(P)=\cap_{\alpha^{\prime}<\alpha}\delta
^{\alpha^{\prime}}(P)$ for any countable limit ordinal $\alpha$. The
\emph{rank} of $\delta$ is defined to be the smallest countable ordinal
$\alpha$ such that $\delta^{\alpha}(X)=\emptyset$ if such an $\alpha$ exists,
and $\omega_{1}$ otherwise.

The oscillation rank $\beta$ is associated with a family of derivations. Let
$\varepsilon>0$ and a function $f:X\rightarrow{\mathbb{R}}$ be given. For any
$F\in{\mathcal{C}}$, let $D^{0}(f,\varepsilon,F)=F$ and $D^{1}(f,\varepsilon
,F)$ be the set of all $x\in F$ such that for any neighborhood $U$ of $x$,
there exist $x_{1},x_{2}\in F\cap U$ such that $\left\vert f(x_{1}%
)-f(x_{2})\right\vert \geq\varepsilon$.

Let $\beta(f,\varepsilon)$ be the rank of the derivation $D^{1}(f,\varepsilon
,\cdot)$. The \emph{oscillation rank} of $f$ is $\beta(f)=\sup_{\varepsilon
>0}\beta(f,\varepsilon)$.

The convergence rank $\gamma$ is defined analogously. Let $\left(
f_{n}\right)  $ be a sequence of real-valued functions and $\varepsilon>0$.
For any $F\in{\mathcal{C}}$, let {$D$}$^{0}(\left(  f_{n}\right)
,\varepsilon,F)=F$ and {$D$}$^{1}(\left(  f_{n}\right)  ,\varepsilon,F)$ be
the set of all $x\in F$ such that for any neighborhood $U$\ and any
$m\in\mathbb{N},$ there are two integers $n_{1}$,\thinspace\ $n_{2}$ with
$n_{1}>n_{2}>m$ and $x^{\prime}\in U\cap F$ such that $\left\vert f_{n_{1}%
}\left(  x^{\prime}\right)  -f_{n_{2}}\left(  x^{\prime}\right)  \right\vert
\geq\varepsilon.$ Let $\gamma\left(  \left(  f_{n}\right)  ,\varepsilon
\right)  $ be the rank of the derivation {$D$}$^{1}(\left(  f_{n}\right)
,\varepsilon,\cdot)$. The \emph{convergence rank} of $\left(  f_{n}\right)  $
is $\gamma(\left(  f_{n}\right)  )=\sup_{\varepsilon>0}\gamma(\left(
f_{n}\right)  ,\varepsilon)$. \ If $f$ is a function of Baire class one, set
$\gamma\left(  f\right)  =\inf\{\gamma(\left(  f_{n}\right)  )\},$where the
infimum is taken over all sequences $\left(  f_{n}\right)  $ of continuous
functions on $X$ converging pointwise to $f.$

If $\tau$ is a Polish topology on $X$, denote the space of Baire class $\xi$
functions on $(X,\tau)$ by ${\mathfrak{B}}_{\xi}(X,\tau),$ or simply
${\mathfrak{B}}_{\xi}(\tau)$. The oscillation rank $\beta$ on ${\mathfrak{B}%
}_{1}(X,\tau)$ is denoted by $\beta_{\tau}$ and its derived sets by $D_{\tau
}^{\eta}(f,\varepsilon,X)$.

Since $\left(  X,\tau\right)  $ is metrizable, every closed set is a
$G_{\delta}$ set. For $1\leq\xi<\omega_{1},$ define the classes $\Sigma_{\xi
}^{0}(\tau)$ and $\Pi_{\xi}^{0}(\tau)$ recursively as follows. Set $\Sigma
_{1}^{0}(\tau)=\tau,$
\[
\Pi_{\xi}^{0}(\tau)=\left\{  F:F^{c}\in\Sigma_{\xi}^{0}(\tau)\right\}  ,
\]
and
\[
\Sigma_{\xi}^{0}\left(  \tau\right)  =\left\{  \cup_{n}A_{n}:A_{n}\in\Pi
_{\xi_{n}}^{0}(\tau),\xi_{n}<\xi\right\}  \text{ if }\xi>1.
\]
In addition, $\Delta_{\xi}^{0}(\tau)=\Sigma_{\xi}^{0}(\tau)\cap\Pi_{\xi}%
^{0}(\tau)$ are called the ambiguous classes.

A rank $\rho$ in ${\mathfrak{B}}_{1}(X)\ $may be extended to a rank $\rho
_{\xi}^{\ast}$ in ${\mathfrak{B}}_{\xi}(X)$ for any countable ordinal $\xi$ in
the following manner. Let $f\in{\mathfrak{B}}_{\xi}(X),$ consider the set of
topologies $T_{f,\xi}=\left\{  \tau^{\prime}:\tau^{\prime}\supseteq\tau\text{
Polish, }\tau^{\prime}\subseteq\Sigma_{\xi}^{0}\left(  \tau\right)
,f\in\mathcal{B}_{1}\left(  \tau^{\prime}\right)  \right\}  ,$ i.e.,
$T_{f,\xi}$ is the set of Polish refinements $\tau^{\prime}\,$\ of the
original topology $\tau$ that are subsets of $\Sigma_{\xi}^{0}\left(
\tau\right)  $ so that $f$ is of Baire class one under $\tau^{\prime}.$ It can
be shown that $T_{f,1}=\left\{  \tau\right\}  $ and $T_{f,\xi}\neq\emptyset$
for every $f\in{\mathfrak{B}}_{\xi}(X)$ (\cite[5.2]{EKV}). Now define
\[
\rho_{\xi}^{\ast}\left(  f\right)  =\min_{\tau^{\prime}\in T_{f,\xi}}%
\rho_{\tau^{\prime}}\left(  f\right)  ,
\]
where $\rho_{\tau^{\prime}}\left(  f\right)  $ is the $\rho$ rank of $f$ in
the $\tau^{\prime}$ topology.

\section{$\beta_{\xi}^{\ast}$ and $\gamma_{\xi}^{\ast}$ are essentially
equivalent}

Let $X$ be a Polish space and let $\mathfrak{B}_{\xi}(X)$ be the space of (not
necessarily bounded) Baire class $\xi$ functions on $X.$ An (\emph{ordinal})
\emph{rank} on $\mathfrak{B}_{\xi}(X)$ is a function $\rho$ from
$\mathfrak{B}_{\xi}(X)$ into the set of ordinals. Two ranks $\rho_{1}$ and
$\rho_{2}$ on ${\mathfrak{B}}_{\xi}(X)$ are said to be \emph{essentially
equivalent} if for any $f\in{\mathfrak{B}}_{\xi}(X)$ and any countable ordinal
$\zeta$, $\rho_{1}(f)\leq\omega^{\zeta}$ if and only $\rho_{2}(f)\leq
\omega^{\zeta}$. Question 5.8 (which is the same as Question 8.6) in
\cite{EKV} asks if the ranks $\beta_{\xi}^{\ast}$ and $\gamma_{\xi}^{\ast}$
are essentially equivalent. The main result of this section is an affirmative
answer to this question. First we consider the case $\xi=1$. In this case, the
question was stated separately in \cite{EKV} as Question 3.41. One half of the
result was already obtained in \cite{EKV}.

\begin{prop}
\label{p2.1} \cite[Theorem 3.24]{EKV} If $f$ is a Baire class $1$ function,
then $\beta(f) \leq\gamma(f)$.
\end{prop}

The next lemma comes from \cite[Lemma 3.1]{LT}. We include the proof for the
reader's convenience.

\begin{lem}
\label{5}Suppose that $F$ is a closed subset of $X$ and that $f$ is a Baire
class $1$ function on $X$. For any $\varepsilon>0$, there exists a continuous
function $g:F\setminus D^{1}(f,\varepsilon,F)\rightarrow\mathbb{R}$ such that
$|g(x)-f(x)|<\varepsilon$ for all $x\in F\backslash D^{1}(f,\varepsilon,F)$.
\end{lem}

\begin{proof}
If $x\in F\setminus D^{1}(f,\varepsilon,F),$ there exists an open neighborhood
$U_{x}$ of $x$ in $F$ such that $U_{x}\cap D^{1}( f,\varepsilon,F) =\emptyset$
and that $\vert f( x_{1}) -f( x_{2}) \vert<\varepsilon\text{ for all }%
x_{1},\,x_{2}\in U_{x}.$ The collection
\[
\mathcal{U}=\{ U_{x}:x\in F\setminus D^{1}( f,\varepsilon,F) \}
\]
is an open cover of the paracompact space $F\setminus D^{1}( f,\varepsilon,F)
$. By \cite{D}, Theorems IX.5.3 and VIII.4.2, there exists a (continuous)
partition of unity $( \varphi_{U}) _{U\in\mathcal{U}}$ subordinated to
$\mathcal{U}.$ If $U=U_{x}\in\mathcal{U}$ for some $x\in F\setminus D^{1}(
f,\varepsilon,F) $, let $a_{U}=f(x)$. Define $g$ on $F\setminus D^{1}(
f,\varepsilon,F)$ by $g=\sum_{U\in\mathcal{U}}a_{U}\varphi_{U}.$ $g$ is a
well-defined continuous function on $F\backslash D^{1}(f,\varepsilon,F)$ since
$\{ \text{supp\thinspace}\varphi_{U}:U\in\mathcal{U}\} $ is locally finite.
Let $x\in F\setminus D^{1}( f,\varepsilon,F) .$ Then $\mathcal{V}=\{
U\in\mathcal{U}:\varphi_{U}( x) \neq0\} $ is a finite set, $\varphi_{U}( x)
>0$ for all $U\in\mathcal{V}$ and $\sum_{U\in\mathcal{V}}\varphi_{U}( x) =1.$
If $U\in\mathcal{V}$, let $U=U_{y}$ for some $y\in F\setminus D^{1}(
f,\varepsilon,F) .$ Since $x,y\in U_{y}$, $\vert a_{U}-f( x) \vert=\vert f( y)
-f( x) \vert<\varepsilon.$ It follows that
\begin{align*}
\vert g( x) -f( x) \vert &  =\vert\sum_{U\in\mathcal{U}}a_{U}\varphi_{U}( x)
-f( x) \vert=\vert\sum_{U\in\mathcal{V}}a_{U}\varphi_{U}( x) -\sum
_{U\in\mathcal{V}}f( x) \varphi_{U}( x) \vert\\
&  \leq\sum_{U\in\mathcal{V}}\vert a_{U}-f( x) \vert\varphi_{U}( x)
<\varepsilon.
\end{align*}

\end{proof}

In the proof of the next proposition, we fix a particular metric $d$ that
generates the topology on $X$.

\begin{prop}
\label{p2.3} Let $f\in\mathfrak{B}_{1}(X)$ with $\beta( f) \leq\beta_{0}.$ For
all $\varepsilon>0,$ there exists $g\in\mathfrak{B}_{1}(X)$ such that
$\gamma(g) \leq\beta_{0}$ and that $|f(x)-g(x)| <\varepsilon$ for all $x\in X$.
\end{prop}

\begin{proof}
Let $\varepsilon>0$ be given. For all $\alpha<\beta_{0},$ apply Lemma \ref{5}
to $F=D^{\alpha}( f,\varepsilon,X)$ to obtain a continuous function
\[
g_{\alpha}:D^{\alpha}( f,\varepsilon,X) \setminus D^{\alpha+1}(f,\varepsilon
,X) \to\mathbb{R}%
\]
such that $|g_{\alpha}(x) -f(x)| < \varepsilon$ for all $x\in D^{\alpha}(
f,\varepsilon,X) \setminus D^{\alpha+1}(f,\varepsilon,X)$. Let $g=\cup
_{\alpha< \beta_{0}}g_{\alpha}.$ Since $D^{\beta_{0}}(f,\varepsilon,X) =
\emptyset$, $g$ is defined on all of $X$. For any $n\in\mathbb{N}$ and
$\alpha<\beta_{0}$, let $U_{n}^{\alpha}$ be the $\frac{1}{n}$-neighborhood
(with respect to the metric $d$) of $D^{\alpha}( f,\varepsilon,X) $. Consider
the set
\[
Y_{n}=\cap_{\alpha<\beta_{0}}[( U_{n}^{\alpha}) ^{c}\cup D^{\alpha}(
f,\varepsilon,X)].
\]
Clearly $Y_{n}$ is closed and $Y_{n}\subseteq Y_{n+1}.$ We claim that $g$ is
continuous on $Y_{n}.$ Let $x\in Y_{n}$ and let $( x_{k}) $ be a sequence in
$Y_{n}$ converging to $x.$ We wish to show that $(g(x_{k}))$ converges to
$g(x)$. Since $D^{\beta_{0}}(f,\varepsilon,X) = \emptyset$, there exists
$\alpha_{0} < \beta_{0}$ such that $x\in D^{\alpha_{0}}( f,\varepsilon,X)
\setminus D^{\alpha_{0}+1}( f,\varepsilon,X).$ If $x_{k} \notin D^{\alpha_{0}%
}( f,\varepsilon,X)$, then $x_{k} \in( U_{n}^{\alpha_{0}}) ^{c}$ by definition
of $Y_{n}$. As $x$ lies outside of the closed set $( U_{n}^{\alpha_{0}}) ^{c}%
$, this is only possible for finitely many $k$. Without loss of generality, we
may assume that $x_{k} \in D^{\alpha_{0}}(f,\varepsilon,X)$ for all $k$.
Similarly, since $x$ lies outside of the closed set $D^{\alpha_{0}+1}(
f,\varepsilon,X)$, we may also assume that $x_{k} \notin D^{\alpha_{0}+1}(
f,\varepsilon,X)$ for all $k$. Therefore, $x, x_{k} \in W =D^{\alpha_{0}}(
f,\varepsilon,X)\backslash D^{\alpha_{0}+1}( f,\varepsilon,X)$ for all $k$.
Since $g = g_{\alpha_{0}}$ on $W$, and $g_{\alpha_{0}}$ is continuous on $W$,
it follows that $(g(x_{k}))$ converges to $g(x)$, as claimed. This shows that
$g_{|Y_{n}}$ is a continuous function on $Y_{n}$ for each $n$.

Extend $g_{|Y_{n}}$ to a continuous function $h_{n}$ on $X.$ We claim that
$(h_{n})$ converges to $g$ pointwise. Given $x\in X,$ choose $\alpha<
\beta_{0}$ so that $x\in D^{\alpha}( f,\varepsilon,X) \setminus D^{\alpha+1}(
f,\varepsilon,X)$. Let $n_{x}$ be an integer so that
\[
d( x,D^{\alpha+1}( f,\varepsilon,X) ) >\frac{1}{n_{x}}.
\]
In particular, $x\in(U^{\alpha+1}_{n_{x}})^{c}$. Consider any $n\geq n_{x}$.
If $\alpha^{\prime}>\alpha$, then $D^{\alpha^{\prime}}(f,\varepsilon,X)
\subseteq D^{\alpha+1}(f,\varepsilon,X)$ and hence $U^{\alpha^{\prime}}_{n}
\subseteq U^{\alpha+1}_{n}\subseteq U^{\alpha+1}_{n_{x}}$. Thus $x\in
(U^{\alpha^{\prime}}_{n})^{c}$. On the other hand, if $\alpha^{\prime\prime
}\leq\alpha$, then $x\in D^{\alpha}( f,\varepsilon,X) \subseteq D^{\alpha
^{\prime\prime}}(f,\varepsilon,X)$. Therefore, $x\in Y_{n}$ and hence $h_{n}(
x) =g( x) $ for all $n>n_{x}.$ Evidently, $(h_{n}) $ converges pointwise to
$g.$

Finally, let us show that for any $\eta>0,$ $D^{\alpha}( ( h_{n}) ,\eta,X)
\subseteq D^{\alpha}( f,\varepsilon,X) .$ Since $\beta(f)\leq\beta_{0}$, it
would follow that $\gamma(g) \leq\gamma((h_{n}))\leq\beta_{0}$. The proof is
by induction on $\alpha$. A moment's reflection shows that it is sufficient to
establish the claim for successor ordinals. Assume that the claim holds for
some $\alpha$. Suppose, if possible, that there exists $x\in D^{\alpha+1}( (
h_{n}) ,\eta,X) \setminus D^{\alpha+1}( f,\varepsilon,X) .$ Let $N\in
\mathbb{N}$ be such that $x\notin U_{N}^{\alpha+1}.$ By definition of
$D^{\alpha+1}( ( h_{n}) ,\eta,X)$, there exist $y \in D^{\alpha}((h_{n}%
),\eta,X)$, $d(y,x) < \frac{1}{2N}$, and $m,n > 2N$ so that $|h_{m}(y)
-h_{n}(y)| \geq\eta$. By the inductive hypothesis, $y \in D^{\alpha
}(f,\varepsilon, X)$. Since $d(x,D^{\alpha+1}(f,\varepsilon,X)) > \frac{1}{N}$
and $d(y,x)< \frac{1}{2N}$, $y \notin U^{\alpha+1}_{2N}$. Thus, if
$\alpha^{\prime}\leq\alpha$, then $y \in D^{\alpha}(f,\varepsilon,X) \subseteq
D^{\alpha^{\prime}}(f,\varepsilon,X)$; while if $\alpha^{\prime\prime}>
\alpha$, then $y \in(U^{\alpha+1}_{2N})^{c}\subseteq(U^{\alpha^{\prime\prime}%
}_{2N})^{c}$. This proves that $y \in Y_{2N}$. By definition,
\[
\vert h_{m}(y) -h_{n}( y) \vert=\vert g( y) -g( y) \vert=0<\eta,
\]
contrary to the choices of $m,n$ and $y$.
\end{proof}

Given two ordinals $\xi_{1}$ and $\xi_{2}$, we write $\xi_{1} \lesssim\xi_{2}$
if $\xi_{2}\leq\omega^{\eta}$ implies $\xi_{1}\leq\omega^{\eta}$. If both
$\xi_{1} \lesssim\xi_{2}$ and $\xi_{2}\lesssim\xi_{1}$ hold, then we write
$\xi_{1} \approx\xi_{2}$.

\begin{prop}
\cite[Proposition 3.34]{EKV}\label{6} If the sequence $(f_{n})$ of Baire class
1 functions converges uniformly to $f,$ then $\gamma(f)\lesssim\sup
\{\gamma(f_{n})\}$.
\end{prop}

The next theorem solves Question 3.41 in \cite{EKV} in the affirmative.

\begin{thm}
\label{7} The ranks $\beta$ and $\gamma$ are essentially equivalent. That is
$\beta(f) \approx\gamma(f)$ for any $f\in{\mathfrak{B}}_{1}(X)$.
\end{thm}

\begin{proof}
By Proposition \ref{p2.1}, $\beta(f) \leq\gamma(f)$ for all $f\in
{\mathfrak{B}}_{1}(X)$. Conversely, suppose that $\beta(f) \leq\omega^{\eta}$.
By Proposition \ref{p2.3}, there is a sequence $(g_{n})$ in ${\mathfrak{B}%
}_{1}(X)$ such that $\gamma(g_{n}) \leq\omega^{\eta}$ for each $n$ and that
$(g_{n})$ converges to $f$ uniformly on $X$. By Proposition \ref{6},
$\gamma(g) \lesssim\sup\{\gamma(g_{n})\} \leq\omega^{\eta}$.
\end{proof}

\noindent\textbf{Remark}. If $X$ is compact metric, then $\beta(f) =
\gamma(f)$ for all $f\in{\mathfrak{B}}_{1}(X)$; see \cite[Theorem 5.5]{LT}. It
is not known if this continues to hold for general Polish spaces $X$. This is
part of Question 8.1 in \cite{EKV}.

\begin{cor}
\label{c2.6} The ranks $\beta_{\xi}^{\ast}$ and $\gamma_{\xi}^{\ast}$ are
essentially equivalent for any countable ordinal $\xi$.
\end{cor}

\begin{proof}
Let $f\in\mathfrak{B}_{\xi}( \tau) .$ If $\beta_{\xi}^{\ast}( f) \leq
\omega^{\zeta}$ for some $\zeta<\omega_{1}.$ Then there exists $\tau^{\prime
}\in T_{f,\xi}$ such that $\beta_{\tau^{\prime}}( f) =\beta_{\xi}^{\ast}( f)
\leq\omega^{\zeta}. $ It follows from Theorem \ref{7} that $\gamma_{\xi}%
^{\ast}( f) \leq\gamma_{\tau^{\prime}}( f) \leq\omega^{\zeta}.$ Thus
$\gamma_{\xi}^{\ast}( f) \lesssim\beta_{\xi}^{\ast}( f) .$ The reverse
inequality follows similarly from \cite[Corollary 5.6]{EKV}.
\end{proof}

\section{$\beta_{\xi}^{\ast}$ is not essentially multiplicative}

In this section, we show that the rank $\beta^{*}_{\xi}$ is not essentially
multiplicative. That is, there are a Polish space $X$ and functions $f,g
\in{\mathfrak{B}}_{\xi}(X)$ so that $\beta^{*}_{\xi}(fg) > \max\{\beta
^{*}_{\xi}(f), \beta^{*}_{\xi}(g)\}$. Since $\beta^{*}_{\xi}$ and $\gamma
^{*}_{\xi}$ are essentially equivalent by Corollary \ref{c2.6}, $\gamma
^{*}_{\xi}$ is also not essentially multiplicative. This answers Question 3.32
and Question 5.16 (= Question 8.4) from \cite{EKV} in the negative.

For the rest of the section, let $\tau$ be a Polish topology on $X$ and let
$\xi$ be a countable ordinal $>1$. The space of Baire class $\xi$ functions on
$(X,\tau)$ is denoted by ${\mathfrak{B}}_{\xi}(\tau)$. The ordinal rank
$\beta$ on ${\mathfrak{B}}_{1}(X,\tau)$ is denoted by $\beta_{\tau}$ and its
derived sets by $D_{\tau}^{\eta}(f,\varepsilon,X)$.

\begin{lem}
\label{l1} Let $(G_{n})$ be a sequence in $\Sigma_{\xi}^{0}(\tau)$ and let
$\tau^{\prime}$ be a Polish topology on $X$ so that $\tau\subseteq\tau
^{\prime}\subseteq\Sigma_{\xi}^{0}(\tau)$. There is a Polish topology
$\tau^{\prime\prime}$ containing $\tau^{\prime}\cup(G_{n})$ such that
$\tau^{\prime\prime}\subseteq\Sigma_{\xi}^{0}(\tau)$.
\end{lem}

\begin{proof}
Let $(B_{n})$ be a countable basis of the topology $\tau^{\prime}$. For each
$n$, $B_{n}$ and $G_{n}$ are countable unions of sets in $\cup_{\eta<\xi}%
\Pi_{\eta}^{0}(\tau)\subseteq\Delta_{\xi}^{0}(\tau)$. By Kuratowski's Theorem
\cite[Theorem 22.18]{K}, there is a Polish topology $\tau^{\prime\prime
}\subseteq\Sigma_{\xi}^{0}(\tau)$ such that all $G_{n}$ and $B_{n}$ are
countable unions of sets in $\Delta_{1}^{0}(\tau^{\prime\prime})$. In
particular, $G_{n},B_{n}\in\tau^{\prime\prime}$ for all $n$. Thus
$\tau^{\prime\prime}$ contains $\tau^{\prime}\cup(G_{n})$.
\end{proof}

\begin{prop}
\label{p2} Let $\zeta$ and $\eta$ be countable ordinals so that $\zeta$ is
limit and that $\zeta^{\prime}\cdot\eta\leq\zeta$ for any $\zeta^{\prime
}<\zeta$. Let $\tau^{\prime}$ be a Polish topology on $X$ so that
$\tau\subseteq\tau^{\prime}\subseteq\Sigma_{\xi}^{0}(\tau)$. Suppose that
$\chi_{A}\in{\mathfrak{B}}_{1}(\tau^{\prime})$ and that $\beta_{\tau^{\prime}%
}(\chi_{A})\leq\zeta\cdot\eta$. Then there exist a Polish topology
$\tau^{\prime\prime}$ on $X$ so that $\tau^{\prime}\subseteq\tau^{\prime
\prime}\subseteq\Sigma_{\xi}^{0}(\tau)$ and a function $g:X\to{\mathbb{N}}$
such that $\beta_{\tau^{\prime\prime}}(\frac{\chi_{A}}{g})\leq\zeta$ and
$\beta_{\tau^{\prime\prime}}(g)\leq\eta$.
\end{prop}

\begin{proof}
Abbreviate $D_{\tau^{\prime}}^{\rho}(\chi_{A},1,X)$ as $D_{\tau^{\prime}%
}^{\rho}$. If $0\leq\theta<\eta$, then
\[
(D_{\tau^{\prime}}^{\zeta\cdot(\theta+1)})^{c}=\bigcup_{0<\rho<\zeta}%
(D_{\tau^{\prime}}^{\zeta\cdot\theta+\rho})^{c}.
\]
Each set $(D_{\tau^{\prime}}^{\zeta\cdot\theta+\rho})^{c}$ is $\tau^{\prime}%
$-open and thus in $\Sigma_{\xi}^{0}(\tau)$. Since the class $\Sigma^{0}_{\xi
}(\tau)$ has the generalized reduction property \cite[Theorem 22.16]{K}, there
are pairwise disjoint sets $(G_{\rho}^{\theta})_{0<\rho<\zeta}$ in
$\Sigma_{\xi}^{0}(\tau)$ such that $G_{\rho}^{\theta}\subseteq(D_{\tau
^{\prime}}^{\zeta\cdot\theta+\rho})^{c}$ for each $\rho$ and that
$\bigcup_{0<\rho<\zeta}G_{\rho}^{\theta}=(D_{\tau^{\prime}}^{\zeta\cdot
(\theta+1)})^{c}$. By Lemma \ref{l1}, there is a Polish topology $\tau
^{\prime\prime}\subseteq\Sigma_{\xi}^{0}(\tau)$ that contains
\[
\tau^{\prime}\cup\{G_{\rho}^{\theta}:0<\rho<\zeta,\theta<\eta\}.
\]
Since
\[
\cup_{0<\rho<\zeta}(G^{\theta}_{\rho}\cap D^{\zeta\cdot\theta}_{\tau^{\prime}%
}) = D^{\zeta\cdot\theta}_{\tau^{\prime}} \backslash D^{\zeta\cdot(\theta
+1)}_{\tau^{\prime}} \text{ and }D^{\zeta\cdot\eta}_{\tau^{\prime}} =
\emptyset,
\]
the sets $G_{\rho}^{\theta}\cap D^{\zeta\cdot\theta}$, $0<\rho<\zeta$,
$\theta<\eta$, form a partition of $X$. Fix a bijection $i$ from the ordinal
interval $[1,\zeta)$ onto ${\mathbb{N}}$. Define $g:X\to{\mathbb{N}}$ by
$g(x)=i(\rho)$ for $x\in G_{\rho}^{\theta}\cap D^{\zeta\cdot\theta}%
_{\tau^{\prime}}$. For any $\varepsilon>0$, there are only finitely many
$\rho\in\lbrack1,\zeta)$ such that $i(\rho)\leq\frac{1}{\varepsilon}$. Fix
$\rho_{\varepsilon}<\zeta$ so that $i(\rho)>\frac{1}{\varepsilon}$ for all
$\rho>\rho_{\varepsilon}$.

\medskip

\noindent\underline{Claim}. For any $\varepsilon>0$ and any $\theta\leq\eta$,
\begin{equation}
D_{\tau^{\prime\prime}}^{\theta}({g},\varepsilon,X)\cup D_{\tau^{\prime\prime
}}^{(1+\rho_{\varepsilon})\cdot\theta}(\frac{\chi_{A}}{g},\varepsilon
,X)\subseteq D_{\tau^{\prime}}^{\zeta\cdot\theta}. \label{e1}%
\end{equation}

\medskip\noindent We prove the claim by induction on $\theta$. The case
$\theta=0$ is trivial. Suppose that it is true for some $\theta<\eta$. Let
\[
P=D_{\tau^{\prime\prime}}^{\theta}({g},\varepsilon,X)\text{ and }%
Q=D_{\tau^{\prime\prime}}^{(1+\rho_{\varepsilon})\cdot\theta}(\frac{{\chi}%
_{A}}{g},\varepsilon,X).
\]
Then $P\cup Q\subseteq D_{\tau^{\prime}}^{\zeta\cdot\theta}$ and the latter
set is $\tau^{\prime}$- and thus $\tau^{\prime\prime}$-closed. Hence
\begin{align}
D_{\tau^{\prime\prime}}^{\theta+1}({g},\varepsilon,X)  &  =D_{\tau
^{\prime\prime}}({g},\varepsilon,P)\subseteq D_{\tau^{\prime\prime}}%
({g},\varepsilon,D_{\tau^{\prime}}^{\zeta\cdot\theta})\text{ and }\label{e2}\\
D_{\tau^{\prime\prime}}^{(1+\rho_{\varepsilon})\cdot(\theta+1)}(\frac{\chi
_{A}}{g},\varepsilon,X)  &  =D_{\tau^{\prime\prime}}^{1+\rho_{\varepsilon}%
}(\frac{\chi_{A}}{g},\varepsilon,Q)\subseteq D_{\tau^{\prime\prime}}%
^{1+\rho_{\varepsilon}}(\frac{\chi_{A}}{g},\varepsilon,D_{\tau^{\prime}%
}^{\zeta\cdot\theta}). \label{e3}%
\end{align}
For $0<\rho<\zeta$, $G_{\rho}^{\theta}\cap D_{\tau^{\prime}}^{\zeta\cdot
\theta}$ is relatively $\tau^{\prime\prime}$-open in $D_{\tau^{\prime}}%
^{\zeta\cdot\theta}$. Furthermore, $g$ is constant on this set while
$\frac{\chi_{A}}{g}$ takes values in the set $\{0,\frac{1}{i(\rho)}\}$ there.
Hence
\begin{align*}
D_{\tau^{\prime\prime}}({g},\varepsilon,D_{\tau^{\prime}}^{\zeta\cdot\theta
})\cap G_{\rho}^{\theta}  &  =\emptyset\text{ and }\\
D_{\tau^{\prime\prime}}(\frac{\chi_{A}}{g},\varepsilon,D_{\tau^{\prime}%
}^{\zeta\cdot\theta})\cap G_{\rho}^{\theta}  &  =\emptyset\text{ if
$i(\rho)>\frac{1}{\varepsilon}$.}%
\end{align*}
Therefore,
\begin{align}
D_{\tau^{\prime\prime}}({g},\varepsilon,D_{\tau^{\prime}}^{\zeta\cdot\theta})
&  \subseteq D_{\tau^{\prime}}^{\zeta\cdot(\theta+1)}\text{ and }\label{e4}\\
Q^{\prime}=D_{\tau^{\prime\prime}}(\frac{\chi_{A}}{g},\varepsilon
,D_{\tau^{\prime}}^{\zeta\cdot\theta})  &  \subseteq\bigcup_{\rho\leq
\rho_{\varepsilon}}(G_{\rho}^{\theta}\cap D_{\tau^{\prime}}^{\zeta\cdot\theta
})\cup D_{\tau^{\prime}}^{\zeta\cdot(\theta+1)}\label{e5}\\
&  \subseteq(D_{\tau^{\prime}}^{\zeta\cdot\theta}\backslash D_{\tau^{\prime}%
}^{\zeta\cdot\theta+\rho_{\varepsilon}})\cup D_{\tau^{\prime}}^{\zeta
\cdot(\theta+1)}.\nonumber
\end{align}
From (\ref{e2}) and (\ref{e4}), we see that
\begin{equation}
D_{\tau^{\prime\prime}}^{\theta+1}({g},\varepsilon,X)\subseteq D_{\tau
^{\prime}}^{\zeta\cdot(\theta+1)}. \label{e6}%
\end{equation}
Since $g(x)\geq1$ for all $x$, $D_{\tau^{\prime\prime}}^{\rho_{\varepsilon}%
}(\frac{\chi_{A}}{g},\varepsilon,Q^{\prime})\subseteq D_{\tau^{\prime\prime}%
}^{\rho_{\varepsilon}}(\chi_{A},1,Q^{\prime})$. Note also that $Q^{\prime
}\subseteq D_{\tau^{\prime}}^{\zeta\cdot\theta}$ and $\tau^{\prime}%
\subseteq\tau^{\prime\prime}$. Thus
\begin{align*}
D_{\tau^{\prime\prime}}^{1+\rho_{\varepsilon}}  &  (\frac{\chi_{A}}%
{g},\varepsilon,D_{\tau^{\prime}}^{\zeta\cdot\theta})=D_{\tau^{\prime\prime}%
}^{\rho_{\varepsilon}}(\frac{\chi_{A}}{g},\varepsilon,Q^{\prime})\\
&  \subseteq D_{\tau^{\prime\prime}}^{\rho_{\varepsilon}}(\chi_{A}%
,1,Q^{\prime})\subseteq D_{\tau^{\prime}}^{\rho_{\varepsilon}}(\chi
_{A},1,D_{\tau^{\prime}}^{\zeta\cdot\theta})\cap Q^{\prime}\\
&  =D_{\tau^{\prime}}^{\zeta\cdot\theta+\rho_{\varepsilon}}\cap Q^{\prime
}\subseteq D_{\tau^{\prime}}^{\zeta\cdot(\theta+1)}\text{ by (\ref{e5}).}%
\end{align*}
In combination with (\ref{e3}), this gives
\begin{equation}
D_{\tau^{\prime\prime}}^{(1+\rho_{\varepsilon})\cdot(\theta+1)}(\frac{\chi
_{A}}{g},\varepsilon,X)\subseteq D_{\tau^{\prime}}^{\zeta\cdot(\theta+1)}.
\label{e7}%
\end{equation}
Together, (\ref{e6}) and (\ref{e7}) yield the claim for $\theta+1$.

Finally, consider the case where $\theta\leq\eta$ is a limit ordinal and
assume that the claim has been verified for all $\theta^{\prime}<\theta$.
Then
\begin{align*}
D_{\tau^{\prime\prime}}^{\theta}({g},\varepsilon,X)  &  \cup D_{\tau
^{\prime\prime}}^{(1+\rho_{\varepsilon})\cdot\theta}(\frac{\chi_{A}}%
{g},\varepsilon,X)\\
&  =\bigcap_{\theta^{\prime}<\theta}D_{\tau^{\prime\prime}}^{\theta^{\prime}%
}({g},\varepsilon,X)\cup\bigcap_{\theta^{\prime}<\theta}D_{\tau^{\prime\prime
}}^{(1+\rho_{\varepsilon})\cdot{\theta^{\prime}}}(\frac{\chi_{A}}{{g}%
},\varepsilon,X)\\
&  \subseteq\bigcap_{\theta^{\prime}<\theta}D_{\tau^{\prime}}^{\zeta
\cdot\theta^{\prime}}=D_{\tau^{\prime}}^{\zeta\cdot\theta}.
\end{align*}
This completes the proof of the claim.

\medskip

For a given $\varepsilon>0$, apply the claim with $\theta=\eta$. Since
$D_{\tau^{\prime}}^{\zeta\cdot\eta}=\emptyset$,
\[
D_{\tau^{\prime\prime}}^{\eta}({g},\varepsilon,X)=\emptyset=D_{\tau
^{\prime\prime}}^{(1+\rho_{\varepsilon})\cdot{\eta}}(\frac{\chi_{A}}%
{g},\varepsilon,X).
\]
The first part implies that $\beta_{\tau^{\prime\prime}}(g)\leq\eta$. For the
second part, note that $\rho_{\varepsilon}<\zeta$ and $\zeta$ is limit. Hence
$1+\rho_{\varepsilon}<\zeta$. By assumption, $(1+\rho_{\varepsilon})\cdot
\eta\leq\zeta$. Thus
\[
D_{\tau^{\prime\prime}}^{\zeta}(\frac{\chi_{A}}{g},\varepsilon,X)\subseteq
D_{\tau^{\prime\prime}}^{(1+\rho_{\varepsilon})\cdot{\eta}}(\frac{\chi_{A}}%
{g},\varepsilon,X)=\emptyset.
\]
It follows that $\beta_{\tau^{\prime\prime}}(\frac{\chi_{A}}{g})\leq\zeta$.
\end{proof}

\begin{lem}
\label{l4} Let $\tau^{\prime}$ be a Polish topology on $X$. Suppose that
$\chi_{A}\in{\mathfrak{B}}_{1}(\tau^{\prime})$ and $G\in\tau^{\prime}$. If
$\tau^{\prime\prime}$ is a Polish topology containing $\tau^{\prime}$, then
$G\cap D_{\tau^{\prime\prime}}^{\rho}(\chi_{A},1,X)=G\cap D_{\tau
^{\prime\prime}}^{\rho}(\chi_{A\cap G},1,X)$ for any countable ordinal $\rho$.
\end{lem}

\begin{proof}
The lemma clearly holds for $\rho=0$. Suppose that it holds for some countable
ordinal $\rho$. Let $x\in G\cap D_{\tau^{\prime\prime}}^{\rho+1}(\chi
_{A},1,X)$. In particular, $x\in G\cap D_{\tau^{\prime\prime}}^{\rho}(\chi
_{A},1,X)\subseteq D_{\tau^{\prime\prime}}^{\rho}(\chi_{A\cap G},1,X)$. If
$x\in U\in\tau^{\prime\prime}$, then $x\in U\cap G\in\tau^{\prime\prime}$.
Thus there are $y,z\in U\cap G\cap D_{\tau^{\prime\prime}}^{\rho}(\chi
_{A},1,X)$ such that $|\chi_{A}(y)-\chi_{A}(z)|\geq1$. Without loss of
generality, we may assume that $y\in A$ and $z\notin A$. But then $y\in A\cap
G$ and $z\notin A\cap G$. Hence $|\chi_{A\cap G}(y)-\chi_{A\cap G}(z)|\geq1$.
Since
\[
y,z\in U \cap G\cap D_{\tau^{\prime\prime}}^{\rho}(\chi_{A},1,X)\subseteq U
\cap D_{\tau^{\prime\prime}}^{\rho}(\chi_{A\cap G},1,X),
\]
we conclude that $x\in D_{\tau^{\prime\prime}}^{\rho+1}(\chi_{A\cap G},1,X)$.
This proves that
\[
G\cap D_{\tau^{\prime\prime}}^{\rho+1}(\chi_{A},1,X)\subseteq G\cap
D_{\tau^{\prime\prime}}^{\rho+1}(\chi_{A\cap G},1,X).
\]
Similarly, assume that $x\in G\cap D_{\tau^{\prime\prime}}^{\rho+1}%
(\chi_{A\cap G},1,X)$. In particular, $x\in D_{\tau^{\prime\prime}}^{\rho
}(\chi_{A},1,X)$. For any $U\in\tau^{\prime\prime}$ such that $x\in U$, there
are $y,z\in U\cap G\cap D_{\tau^{\prime\prime}}^{\rho}(\chi_{A\cap
G},1,X)\subseteq U\cap D_{\tau^{\prime\prime}}^{\rho}(\chi_{A},1,X)$ such that
$y\in A\cap G$ and $z\notin A\cap G$. Thus $y\in A$ and $z\notin A$. This
proves that $x\in D_{\tau^{\prime\prime}}^{\rho+1}(\chi_{A},1,X)$. Therefore,
$G\cap D_{\tau^{\prime\prime}}^{\rho+1}(\chi_{A},1,X)\supseteq G\cap
D_{\tau^{\prime\prime}}^{\rho+1}(\chi_{A\cap G},1,X)$.

Suppose that $\rho$ is a countable limit ordinal and that the lemma holds for
any $\rho^{\prime}<\rho$. Then
\begin{align*}
G\cap D_{\tau^{\prime\prime}}^{\rho}(\chi_{A},1,X)  &  =\bigcap_{\rho^{\prime
}<\rho}(G\cap D_{\tau^{\prime\prime}}^{\rho^{\prime}}(\chi_{A},1,X))\\
&  =G\cap\bigcap_{\rho^{\prime}<\rho}D_{\tau^{\prime\prime}}^{\rho^{\prime}%
}(\chi_{A\cap G},1,X)=G\cap D_{\tau^{\prime\prime}}^{\rho}(\chi_{A\cap
G},1,X).
\end{align*}

\end{proof}

\begin{prop}
\label{p5} Let $(X,\tau)$ be an uncountable Polish space and let $\zeta$ be a
nonzero countable ordinal. Then there exists a set $A$ such that $\chi_{A}%
\in{\mathfrak{B}}_{\xi}(\tau)$ and that $\zeta<\beta_{\xi}^{\ast}(\chi
_{A})\leq\zeta+2$.
\end{prop}

\begin{proof}
By \cite[Theorem 5.20]{EKV}, $\beta^{*}_{\xi}$ can take arbitrarily large
values in the set of countable ordinals on characteristic functions in
${\mathfrak{B}}_{\xi}(\tau)$. Let $\alpha$ be the smallest ordinal $>\zeta$
such that there exists $\chi_{A}\in{\mathfrak{B}}_{\xi}(\tau)$ with
$\beta_{\xi}^{\ast}(\chi_{A})=\alpha$. Note that $\alpha$ is at least $2$.
Choose a Polish topology $\tau^{\prime}\subseteq\Sigma_{\xi}^{0}(\tau)$
containing $\tau$ such that $\beta_{\tau^{\prime}}(\chi_{A})=\alpha$.

Let $(\eta_{n})$ be a (not necessarily strictly) increasing sequence of
ordinals so that $\eta_{n}+1<\alpha$ for all $n$ and let $\eta=\sup\eta_{n}$.
Apply Lemma \ref{l4} with $G=G_{n}=D_{\tau^{\prime}}^{\eta_{n}}(\chi
_{A},1,X)^{c}$ and $\tau^{\prime\prime}=\tau^{\prime}$ to obtain that
\[
G_{n}\cap D_{\tau^{\prime}}^{\eta_{n}}(\chi_{A\cap G_{n}},1,X)=G_{n}\cap
D_{\tau^{\prime}}^{\eta_{n}}(\chi_{A},1,X)=\emptyset.
\]
Hence $D_{\tau^{\prime}}^{\eta_{n}}(\chi_{A\cap G_{n}},1,X)\subseteq G_{n}%
^{c}$. Since $\chi_{A\cap G_{n}}=0$ on $G_{n}^{c}$, $\beta_{\tau^{\prime}%
}(\chi_{A\cap G_{n}})\leq\eta_{n}+1<\alpha$. By choice of $\alpha$,
$\beta_{\xi}^{\ast}(\chi_{A\cap G_{n}})\leq\zeta$. Let $\tau_{n}%
\subseteq\Sigma_{\xi}^{0}(\tau)$ be a Polish topology containing $\tau$ such
that $\beta_{\tau_{n}}(\chi_{A\cap G_{n}})\leq\zeta$. There is a countable
sequence of sets in $\Sigma_{\xi}^{0}(\tau)$ that includes a basis for each of
the topologies $\tau^{\prime}$ and $\tau_{n}$, $n\in{\mathbb{N}}$. By Lemma
\ref{l1}, there is a Polish topology $\tau^{\prime\prime}\subseteq\Sigma_{\xi
}^{0}(\tau)$ that contains $\tau^{\prime}$ and $\tau_{n}$ for all
$n\in{\mathbb{N}}$. Applying Lemma \ref{l4} with $G=G_{n}$, we see that
\[
G_{n}\cap D_{\tau^{\prime\prime}}^{\zeta}(\chi_{A},1,X)\subseteq
D_{\tau^{\prime\prime}}^{\zeta}(\chi_{A\cap G_{n}},1,X)\subseteq D_{\tau_{n}%
}^{\zeta}(\chi_{A\cap G_{n}},1,X)=\emptyset.
\]
Therefore,
\begin{equation}
D_{\tau^{\prime\prime}}^{\zeta}(\chi_{A},1,X)\subseteq\cap G_{n}^{c}%
=D_{\tau^{\prime}}^{\eta}(\chi_{A},1,X). \label{e3.1}%
\end{equation}
For the rest of the proof, we consider three cases.

\medskip

\noindent\underline{Case 1}. $\alpha$ is a limit ordinal.

\medskip

\noindent Choose a sequence $(\eta_{n})$ strictly increasing to $\alpha$.
Obviously, $\eta_{n}+1 < \alpha$ for all $n$. By (\ref{e3.1}), $D_{\tau
^{\prime\prime}}^{\zeta}(\chi_{A},1,X)\subseteq D^{\alpha}_{\tau^{\prime}%
}(\chi_{A},1,X)=\emptyset$. This implies that $\alpha=\beta_{\xi}^{\ast}%
(\chi_{A})\leq\zeta$, which is impossible.

\medskip

\noindent\underline{Case 2}. $\alpha=\eta+1$ for a limit ordinal $\eta$.

\medskip

\noindent Choose a sequence $(\eta_{n})$ strictly increasing to $\eta$. By
(\ref{e3.1}), $D_{\tau^{\prime\prime}}^{\zeta}(\chi_{A},1,X)\subseteq
P=D_{\tau^{\prime}}^{\eta}(\chi_{A},1,X)$. Hence
\[
D_{\tau^{\prime\prime}}^{\zeta+1}(\chi_{A},1,X)\subseteq D_{\tau^{\prime
\prime}}(\chi_{A},1,P)\subseteq D_{\tau^{\prime}}(\chi_{A},1,P)=D_{\tau
^{\prime}}^{\alpha}(\chi_{A},1,X)=\emptyset.
\]
Therefore, $\alpha=\beta_{\xi}^{\ast}(\chi_{A})\leq\zeta+1$.

\medskip

\noindent\underline{Case 3}. $\alpha=\eta+2$ for some ordinal $\eta$.

\medskip

\noindent Take $\eta_{n}=\eta$ for all $n$. By (\ref{e3.1}), $D_{\tau
^{\prime\prime}}^{\zeta}(\chi_{A},1,X)\subseteq P=D_{\tau^{\prime}}^{\eta
}(\chi_{A},1,X)$. Hence
\[
D_{\tau^{\prime\prime}}^{\zeta+2}(\chi_{A},1,X)\subseteq D_{\tau^{\prime
\prime}}^{2}(\chi_{A},1,P)\subseteq D_{\tau^{\prime}}^{2}(\chi_{A}%
,1,P)=D_{\tau^{\prime}}^{\alpha}(\chi_{A},1,X)=\emptyset.
\]
Therefore, $\alpha=\beta_{\xi}^{\ast}(\chi_{A})\leq\zeta+2$.
\end{proof}

Let $F$ be a closed set in $(X,\tau)$. If $\tau^{\prime}$ is a topology on
$X$, let $\tau_{F}^{\prime}=\{O\cap F:O\in\tau^{\prime}\}$ be the subspace
topology on $F$ induced by $\tau^{\prime}$. Suppose that $f\in{\mathfrak{B}%
}_{\xi}(X,\tau)$. Then $f_{|F}\in\mathfrak{B}_{\xi}(F,\tau_{F})$. Let
$\beta_{\xi,F}^{\ast}(f_{|F})$ denote the $\beta_{\xi}^{\ast}$-rank of
$f_{|F}$.

\begin{lem}
\label{l5}Let $F$ be a closed set in $X$ and let $A$ be a subset of $F$ so
that $\chi_{A}\in\mathfrak{B}_{\xi}(F)$. Then
\[
\beta_{\xi,F}^{\ast}(\chi_{A|F})\leq\beta_{\xi}^{\ast}(\chi_{A})\leq
1+\beta_{\xi,F}^{\ast}(\chi_{A|F}).
\]

\end{lem}

\begin{proof}
Let $\tau^{\prime}\subseteq$ $\Sigma_{\xi}^{0}(\tau)$ be a Polish topology
containing $\tau$ such that $\beta_{\xi}^{\ast}(\chi_{A})=\beta_{\tau^{\prime
}}(\chi_{A})$. Let $V$ be a $\tau_{F}^{\prime}$-closed set. Then $V$ is
$\tau^{\prime}$-closed and for any $\varepsilon>0$, $D_{\tau_{F}^{\prime}%
}(\chi_{A|F},\varepsilon,V)\subseteq D_{\tau^{\prime}}(\chi_{A},\varepsilon
,V)$. Hence
\[
\beta_{\xi,F}^{\ast}(\chi_{A|F})\leq\beta_{\tau_{F}^{\prime}}(\chi_{A|F}%
)\leq\beta_{\tau^{\prime}}(\chi_{A})=\beta_{\xi}^{\ast}(\chi_{A}).
\]

Conversely, let $\tilde{\tau}\subseteq$ $\Sigma_{\xi}^{0}(\tau_{F})$ be a
Polish topology on $F$ containing $\tau_{F}$ such that $\beta_{\xi,F}^{\ast
}(\chi_{A|F})=\beta_{\tilde{\tau}}(\chi_{A|F}).$ Since $F$ is $\tau$-closed,
it is easy to verify by induction that $\Pi_{\zeta}^{0}(\tau_{F})\subseteq
\Pi_{\zeta}^{0}(\tau)$ for all $\zeta<\omega_{1}$. It follows that
$\Sigma_{\xi}^{0}(\tau_{F})\subseteq\Sigma_{\xi}^{0}(\tau).$ Let $(B_{n})$ be
a countable basis for $\tilde{\tau}.$ Then $(B_{n})\subseteq\Sigma_{\xi}%
^{0}(\tau)$. By Lemma \ref{l1}, there exists a Polish topology $\tau
^{\prime\prime}$ on $X$ containing $\tau\cup(B_{n})$ such that $\tau
^{\prime\prime}\subseteq\Sigma_{\xi}^{0}(\tau).$ Clearly $B_{n}\in\tau
_{F}^{\prime\prime}$ for all $n$ and hence $\tilde{\tau}\subseteq\tau
_{F}^{\prime\prime}$. Also, since $F^{c}\in\tau\subseteq\tau^{\prime\prime}$
and $\chi_{A}=0$ on $F^{c}$, $D_{\tau^{\prime\prime}}(\chi_{A},\varepsilon
,X)\subseteq F$. Thus, for all $\gamma<\omega_{1}$,
\[
D_{\tau^{\prime\prime}}^{1+\gamma}(\chi_{A},\varepsilon,X)\subseteq
D_{\tau^{\prime\prime}}^{\gamma}(\chi_{A},\varepsilon,F)=D_{\tau_{F}%
^{\prime\prime}}^{\gamma}(\chi_{A|F},\varepsilon,F)\subseteq D_{\tilde{\tau}%
}^{\gamma}(\chi_{A|F},\varepsilon,F).
\]
Hence $\beta_{\tau^{\prime\prime}}(\chi_{A})\leq1+\beta_{\tilde{\tau}}%
(\chi_{A|F})$ and thus $\beta_{\xi}^{\ast}(\chi_{A})\leq1+\beta_{\xi,F}^{\ast
}(\chi_{A|F})$.
\end{proof}

\begin{prop}
\label{p6}Let $(X,\tau)$ be an uncountable Polish space and $\xi\geq1$. Then
for all limit ordinals $\zeta<\omega_{1},$ there exists $f\in{\mathfrak{B}%
}_{\xi}(\tau)$ and that $\beta_{\xi}^{\ast}(f)=\zeta$.
\end{prop}

\begin{proof}
Let $\left(  U_{n}\right)  $ be a sequence of pairwise disjoint uncountable
open sets. For each $n$, let $F_{n}$ be an uncountable closed set in $X$
contained in $U_{n}$. Let $\zeta_{n}$ be a sequence of ordinals strictly
increasing to $\zeta.$ Applying Proposition \ref{p5} to $F_{n},$ we obtain for
all $n,$ a set $A_{n}\subseteq F_{n}$ such that $\chi_{A_{n}|F_{n}}%
\in{\mathfrak{B}}_{\xi}(\tau_{F_{n}})$ and $\zeta_{n}<\beta_{\xi,F_{n}}^{\ast
}(\chi_{A_{n}|F_{n}})\leq\zeta_{n}+2.$ By the lemma, $\zeta_{n}<\beta_{\xi
}^{\ast}(\chi_{A_{n}})\leq1+\zeta_{n}+2.$ For each $n,$ choose a Polish
topology $\tau_{n}^{\prime}\subseteq\Sigma_{\xi}^{0}(\tau)$ containing $\tau$
such that $\beta_{\tau_{n}^{\prime}}(\chi_{A_{n}})=\beta_{\xi}^{\ast}%
(\chi_{A_{n}}).$ Let $\tau^{\prime\prime}\subseteq\Sigma_{\xi}^{0}(\tau)$ be a
Polish topology containing all $\tau_{n}^{\prime}$ and let $f=\sum\frac
{\chi_{A_{n}}}{n}.$ By \cite[Theorem 5.10]{EKV}$,$
\[
1+\beta_{\xi}^{\ast}(f)\geq\beta_{\xi}^{\ast}(f\chi_{F_{n}})=\beta_{\xi}%
^{\ast}(\frac{1}{n}\chi_{A_{n}})=\beta_{\xi}^{\ast}(\chi_{A_{n}})\text{ for
all }n.
\]
Therefore, $\beta_{\xi}^{\ast}(f)\geq\zeta.$ On the other hand, for all
$\varepsilon>0,$
\[
D_{\tau^{\prime\prime}}(f,\varepsilon,X)\cap U_{n}\subseteq D_{\tau
_{n}^{\prime\prime}}(\frac{1}{n}\chi_{A_{n}},\varepsilon,X)\cap U_{n}%
=\emptyset\text{ if }n>\frac{1}{\varepsilon}.
\]
It follows that
\[
D_{\tau^{\prime\prime}}(f,\varepsilon,X)\subseteq(\cup_{n>1/\varepsilon}%
U_{n})^{c}=F.
\]
Let $\overline{\zeta}=\max\{\zeta_{n}:n\leq1/\varepsilon\}$. Since
$f=\sum_{n\leq1/\varepsilon}\frac{\chi_{A_{n}}}{n}$ on $F$ and $\cup
_{n\leq\frac{1}{\varepsilon}}F_{n}$ is a $\tau^{\prime\prime}$-closed set
outside of which $\sum_{n\leq1/\varepsilon}\frac{\chi_{A_{n}}}{n}$ vanishes,
\[
D_{\tau^{\prime\prime}}^{3+\overline{\zeta}+2}(f,\varepsilon,X)\subseteq
D_{\tau^{\prime\prime}}^{2+\overline{\zeta}+2}(f,\varepsilon,F)=D_{\tau
^{\prime\prime}}^{1+\overline{\zeta}+2}(\sum_{n\leq1/\varepsilon}\frac{1}%
{n}\chi_{A_{n}},\varepsilon,\cup_{n\leq\frac{1}{\varepsilon}}F_{n}).
\]
As each $F_{n},n\leq\frac{1}{\varepsilon}$, is a $\tau^{\prime\prime}$-clopen
subset of $\cup_{n\leq\frac{1}{\varepsilon}}F_{n}$, the last set is equal to
\[
\cup_{n\leq\frac{1}{\varepsilon}}D_{\tau^{\prime\prime}}^{1+\overline{\zeta
}+2}(\frac{1}{n}\chi_{A_{n}},\varepsilon,F_{n})\subseteq\cup_{n\leq\frac
{1}{\varepsilon}}D_{\tau^{\prime\prime}}^{1+\zeta_{n}+2}(\frac{1}{n}%
\chi_{A_{n}},\varepsilon,F_{n})=\emptyset.
\]
Therefore, $\beta_{\tau^{\prime\prime}}(f,\varepsilon)\leq3+\overline{\zeta
}+2\leq\zeta.$ Since $\varepsilon>0$ is arbitrary, $\beta_{\xi}^{\ast}%
(f)\leq\beta_{\tau^{\prime\prime}}(f)\leq\zeta$. This completes the proof of
the proposition.\bigskip
\end{proof}

\bigskip

\noindent\textbf{Remark}. In an uncountable Polish space $(X,\tau)$, any
nonzero countable ordinal is in the range of $\beta=\beta_{1}^{\ast}$.
Propositions \ref{p5} and \ref{p6} say that this is \textquotedblleft almost"
true for $\xi\geq2$. It would be interesting to find out if it is true exactly.

\bigskip

\noindent\textbf{Question}. Let $(X,\tau)$ be an uncountable Polish space and
let $\xi$ be a countable ordinal $\geq2$. Is it true that for any nonzero
countable ordinal $\zeta$, there exists $f\in{\mathfrak{B}}_{\xi}(\tau)$ such
that $\beta^{*}_{\xi}(f) = \zeta$?

\bigskip

Propositions \ref{p2} and \ref{p5} lead to a negative solution of
\cite[Question 5.16]{EKV}.

\begin{thm}
\label{t6} For any countable ordinal $\xi\geq2$, $\beta_{\xi}^{\ast}$ is not
essentially multiplicative. Specifically, if $(X,\tau)$ is an uncountable
Polish space and $\zeta$ is a nonzero countable ordinal, there are functions
$f,g\in{\mathfrak{B}}_{\xi}(X,\tau)$ so that $\beta_{\xi}^{\ast}(f),\beta
_{\xi}^{\ast}(g)\leq\omega^{\zeta}$ and $\beta_{\xi}^{\ast}(fg)>\omega^{\zeta
}$.
\end{thm}

\begin{proof}
Taking $\omega^{\zeta}$ in place of $\zeta$ in Proposition \ref{p5}, we see
that there exists $\chi_{A}\in{\mathfrak{B}}_{\xi}(X,\tau)$ so that
$\omega^{\zeta}<\beta_{\xi}^{\ast}(\chi_{A})\leq\omega^{\zeta}\cdot2$. Let
$\tau^{\prime}\in\Sigma_{\xi}^{0}(\tau)$ be a Polish topology such that
$\beta_{\xi}^{\ast}(\chi_{A})=\beta_{\tau^{\prime}}\left(  \chi_{A}\right)  .$
Since $\zeta^{\prime}\cdot2\leq\omega^{\zeta}$ for any $\zeta^{\prime}%
<\omega^{\zeta}$, it follows from Proposition \ref{p2} that there exist a
Polish topology $\tau^{\prime\prime}$ on $X$ so that $\tau^{\prime}\subseteq$
$\tau^{\prime\prime}\subseteq\Sigma_{\xi}^{0}(\tau)$ and a function
$g:X\rightarrow{\mathbb{N}}$ such that $\beta_{\tau^{\prime\prime}}(\frac
{\chi_{A}}{g})\leq\omega^{\zeta}$ and $\beta_{\tau^{\prime\prime}}(g)\leq
2\leq\omega^{\zeta}$. In particular, $f=\frac{\chi_{A}}{g}$ and $g$ are
functions in ${\mathfrak{B}}_{1}(\tau^{\prime\prime})\subseteq{\mathfrak{B}%
}_{\xi}(\tau)$. Furthermore, $\beta_{\xi}^{\ast}(f)\leq\beta_{\tau
^{\prime\prime}}(f)\leq\omega^{\zeta}$ and similarly $\beta_{\xi}^{\ast
}(g)\leq\omega^{\zeta}$. Obviously, $\beta_{\xi}^{\ast}(fg)>\omega^{\zeta}$.
\end{proof}

\noindent\textbf{Remark}. It is known that $\beta$ is not essentially
multiplicative; see \cite[Theorem 6.5]{LT}.

\section{Multipliers on small Baire classes}

Since the rank $\beta$ is not essentially multiplicative, it is worthwhile to
take a deeper look at the behavior of $\beta$ with respect to multiplication.
Let $X$ be a Polish space. Following \cite{K-L}, for any countable ordinal
$\alpha$, let ${\mathfrak{B}}_{1}^{\alpha}(X)$ be the set of all functions
$f\in{\mathfrak{B}}_{1}(X)$ so that $\beta(f) \leq\omega^{\alpha}$. The spaces
${\mathfrak{B}}^{\alpha}_{1}(X)$ are called the \emph{small Baire classes} and
each is a vector subspace of ${\mathfrak{B}}_{1}(X)$ that is closed under
uniform convergence of sequences.

\begin{Def}
Let $h\in\mathfrak{B}_{1}( X)$ and let $\kappa$ and $\lambda$ be countable
ordinals. We say that $h$ is a \emph{$( \kappa,\lambda)$-multiplier} if the
product $hf\in\mathfrak{B}_{1}^{\lambda}( X) $ whenever $f\in\mathfrak{B}%
_{1}^{\kappa}( X) .$ The set of all $(\kappa,\lambda)$-multipliers is denoted
by ${\mathcal{M}}( \kappa,\lambda)$.
\end{Def}

The main result of the present section is a characterization of the
multipliers in ${\mathcal{M}}(\kappa,\kappa)$ in terms of certain ordinal
ranks, which will be introduced after a brief discussion on regular
derivations. A derivation $\delta$ is said to be \emph{regular} if (a)
$\delta(P) \subseteq\delta(Q)$ for any closed sets $P$ and $Q$ such that $P
\subseteq Q$, and (b) $\delta(P \cup Q) \subseteq\delta(P) \cup\delta(Q)$ for
any closed sets $P$ and $Q$. The main interest in regular derivations is
encapsulated in the next proposition, which is an abstraction of the content
of the proof of \cite[Lemma 5]{K-L}.

\begin{prop}
\label{p0} Let $\delta_{i}$, $1\leq i\leq3$ be regular derivations. Suppose
that $\delta_{1}(P) \subseteq\delta_{2}(P) \cup\delta_{3}(P)$ for any closed
set $P$ in $X$. Then $\delta^{\omega^{\alpha}}_{1}(P) \subseteq\delta
^{\omega^{\alpha}}_{2}(P) \cup\delta^{\omega^{\alpha}}_{3}(P)$ for any
countable ordinal $\alpha$ and any closed set $P$ in $X$.
\end{prop}

\begin{proof}
[Sketch of proof]The proof is by induction on $\alpha$. We only give the proof
for the successor case. Assume that the proposition holds for some $\alpha$.
Using regularity, it is easy to see that
\[
\delta_{1}^{\omega^{\alpha}\cdot2n}(P) \subseteq\bigcup\delta^{\omega^{\alpha
}}_{i_{1}}\delta^{\omega^{\alpha}}_{i_{2}}\cdots\delta^{\omega^{\alpha}%
}_{i_{2n}}(P),
\]
where $(i_{1},\dots, i_{2n})$ runs over all $\{2,3\}$-valued sequences of
length $2n$. Since one of the numbers $2$ or $3$ has to appear at least $n$
times in the sequence $(i_{1},\dots, i_{2n})$, by regularity again,
$\delta_{1}^{\omega^{\alpha}\cdot2n}(P) \subseteq\delta_{2}^{\omega^{\alpha
}\cdot n}(P) \cup\delta_{3}^{\omega^{\alpha}\cdot n}(P)$. Taking intersection
over all $n\in{\mathbb{N}}$ gives the result for $\alpha+1$.
\end{proof}

However, the derivations associated with the rank $\beta$, $D(f,\varepsilon
,\cdot)$, may not be regular. Therefore, we replace with a family of regular
derivations that yields the same rank as $\beta$. Let $g,h$ be real valued
functions on $X$ and let $\varepsilon>0$ be given. If $P$ is a closed subset
of $X$, let $O_{g}(h,\varepsilon,P)$ be the set of all points $x\in P$ so that
for any open neighborhood $U$ of $x$, there exists $y\in U\cap P$ such that
\[
|h(y)-h(x)|(|g(y)|\vee|g(x)|)\geq\varepsilon.
\]
Then define $\delta_{g}(h,\varepsilon,P)$ to be $\overline{O_{g}%
(h,\varepsilon,P)}$. When $g$ is the constant function $1$, we simply write
$O(h,\varepsilon,P)$ and $\delta(g,\varepsilon,P)$ respectively.

\begin{prop}
\label{p4.3} For any $g,h:X\to{\mathbb{R}}$ and $\varepsilon>0$, $\delta
_{g}(h,\varepsilon,\cdot)$ is a regular derivation. Furthermore, for any
closed set $P$,
\[
\delta(h,\varepsilon,P) \subseteq D(h,\varepsilon,P)\subseteq\delta
(h,\frac{\varepsilon}{2},P).
\]

\end{prop}

\begin{proof}
It is clear that $\delta_{g}(h,\varepsilon,\cdot)$ is a derivation that
satisfies condition (a) in the definition of regularity. Let $P$ and $Q$ be
closed sets and let $x\in O_{g}(h,\varepsilon,P\cup Q)$. Since $X$ is
metrizable, there is a sequence $(y_{n})$ in $P\cup Q$ such that
$|h(y_{n})-h(x)|(|g(y_{n})|\vee|g(x)|) \geq\varepsilon$. By taking a
subsequence if necessary, we may assume that $y_{n}$ belongs to, say, $P$ for
all $n$. In particular, since $P$ is closed, $x\in P$. If $U$ is an open
neighborhood of $x$, then for sufficiently large $n$, $y_{n}\in U \cap P$ and
$|h(y_{n})-h(x)|(|g(y_{n})|\vee|g(x)|) \geq\varepsilon$. Thus $x\in
O_{g}(h,\varepsilon,P)$. This shows that $O_{g}(h,\varepsilon,P\cup
Q)\subseteq O_{g}(h,\varepsilon,P)\cup O_{g}(h,\varepsilon,Q)$. It follows
that $\delta_{g}(h,\varepsilon,\cdot)$ satisfies condition (b) of regularity.

If $x\in O(h,\varepsilon,P)$, then for any open neighborhood $U$ of $x$, there
exists $y\in U\cap P$ so that $|h(y) - h(x)| \geq\varepsilon$. Thus $x\in
D(h,\varepsilon,P)$. Hence $\delta(h,\varepsilon,P) \subseteq D(h,\varepsilon
,P)$. On the other hand, suppose that $x\in D(h,\varepsilon,P)$. Then there
are sequences $(y_{n})$ and $(z_{n})$ in $P$, both converging to $x$, so that
$|h(y_{n}) - h(z_{n})| \geq\varepsilon$ for all $n$. For each $n$, either
$|h(y_{n})-h(x)|$ or $|h(z_{n})-h(x)| \geq\frac{\varepsilon}{2}$. By taking a
subsequence if necessary, we may assume that, say, $|h(y_{n}) - h(x)|
\geq\frac{\varepsilon}{2}$ for all $n$. If $U$ is an open neighborhood of $x$,
then for sufficiently large $n$, $y_{n}\in U \cap P$ and $|h(y_{n})-h(x)|
\geq\frac{\varepsilon}{2}$. Thus $x\in O_{g}(h,\varepsilon,P)\subseteq
\delta(h,\frac{\varepsilon}{2},P)$.
\end{proof}

A particular consequence of Proposition \ref{p4.3} is that a function
$f:X\to{\mathbb{R}}$ satisfies $\beta(f) \leq\beta_{0}$ if and only if
$\delta^{\beta_{0}}(h,\varepsilon, X) = \emptyset$ for all $\varepsilon>0$.
This enables us to work with the regular derivation $\delta(f,\varepsilon
,\cdot)$ in place of $D(f,\varepsilon,\cdot)$.

\begin{prop}
\label{p1} Let $g,h$ be real-valued functions on $X$ and let $\varepsilon> 0$.
For any closed set $P$,
\[
\delta(gh, \varepsilon,P) \subseteq\delta_{h}(g,\frac{\varepsilon}{2},P)
\cup\delta_{g}(h,\frac{\varepsilon}{2},P).
\]
Hence
\[
\delta^{\omega^{\alpha}}(gh, \varepsilon,P) \subseteq\delta^{\omega^{\alpha}%
}_{h}(g,\frac{\varepsilon}{2},P) \cup\delta^{\omega^{\alpha}}_{g}%
(h,\frac{\varepsilon}{2},P)
\]
for any countable ordinal $\alpha$.
\end{prop}

\begin{proof}
Suppose that $x\in O(gh,\varepsilon,P)$. If $x\notin O_{h}(g,\frac
{\varepsilon}{2},P)$, then there exists an open neighborhood $U$ of $x$ such
that
\[
|g(y)-g(x)||h(y)|\leq\frac{\varepsilon}{2}\text{ for all $y\in U\cap P$}.
\]
Let $V$ be an open neighborhood of $x$. There exists $y\in U\cap V\cap P$ such
that
\begin{align*}
\varepsilon &  \leq|(gh)(y)-(gh)(x)| &  & \\
&  \leq|g(y)-g(x)||h(y)|+|h(y)-h(x)||g(x)| &  & \\
&  \leq\frac{\varepsilon}{2}+|h(y)-h(x)||g(x)|. &  &
\end{align*}
Thus $|h(y)-h(x)||g(x)|\geq\frac{\varepsilon}{2}$. This proves that $x\in
O_{g}(h,\frac{\varepsilon}{2},P)$. Therefore, $O(gh,\varepsilon,P)\subseteq
O_{h}(g,\frac{\varepsilon}{2},P)\cup O_{g}(h,\frac{\varepsilon}{2},P)$. Taking
closures yield the first statement of the proposition. The second statement
follows from Proposition \ref{p0}.
\end{proof}

Proposition \ref{p1} splits $\delta^{\omega^{\alpha}}(gh,\varepsilon,P)$ into
two parts. We now introduce a derivation to control the first part. If $A$ is
a subset of $X$, denote by $A^{\prime}$ the set of accumulation points of $A$
in $X$. Suppose that $h$ is real valued function on $X$ and $M$ is a
nonnegative real number. Define the derivation $h_{M}$ by $h_{M}%
(P)=(P\cap\{|h|>M\})^{\prime}$. It is easy to check that $h_{M}$ is a regular derivation.

\begin{prop}
\label{p2'} Let $g,h$ be real-valued functions on $X$ and let $\varepsilon>0$.
If $0<M<\infty$ and $\alpha$ is a countable ordinal , then
\[
\delta_{h}^{\omega^{\alpha}}(g,\varepsilon,P)\subseteq h_{M}^{\omega^{\alpha}%
}(P)\cup\delta^{\omega^{\alpha}}(g,\frac{\varepsilon}{M},P)\text{ for any
closed set $P$ in $X$}.
\]

\end{prop}

\begin{proof}
Using Proposition \ref{p0}, we see that it suffices to prove the present
proposition for $\alpha=0$. Assume that $x\in O_{h}(g,\varepsilon,P)\backslash
h_{M}(P)$. There exists an open neighborhood $U$ of $x$ such that $|h(y)|\leq
M$ for all $y\in U\cap P$. For any open neighborhood $V$ of $x$, there exists
$y\in U\cap V\cap P$ so that
\[
\varepsilon\leq|g(y)-g(x)|(|h(y)|\vee|h(x)|)\leq M|g(y)-g(x)|.
\]
Hence $|g(y)-g(x)|\geq\frac{\varepsilon}{M}$. This proves that $x\in
O(g,\frac{\varepsilon}{M},P)$. Therefore, $O_{h}(g,\varepsilon,P)\subseteq
h_{M}(P)\cup O(g,\frac{\varepsilon}{M},P)$. Taking closures complete the proof.
\end{proof}

\begin{cor}
\label{c3} Let $h$ be a real-valued function on $X$ and let $\kappa$ be a
countable ordinal. Assume that there exists $0<M< \infty$ such that
$h_{M}^{\omega^{\kappa}}(X) = \emptyset$. For any $g \in{\mathfrak{B}}%
^{\kappa}_{1}(X)$ and any $\varepsilon> 0$, $\delta_{h}^{\omega^{\kappa}%
}(g,\varepsilon,P)= \emptyset$ for any closed set $P$ in $X$.
\end{cor}

Control over the second part of $\delta^{\omega^{\alpha}}(gh, \varepsilon,P)$
is much more delicate. We begin with a lemma.

\begin{lem}
\label{l4'} Let $g,h$ be real-valued functions on $X$ and let $\eta>0$. Let
$P$ be a closed set in $X$. If $U$ is an open set such that $M=\sup_{x\in
U\cap P}|g(x)|<\infty$, then $U\cap\delta_{g}^{\alpha}(h,\eta,P)\subseteq
\delta^{\alpha}(h,\frac{\eta}{M},P)$ for any $\alpha<\omega_{1}$.
\end{lem}

\begin{proof}
We prove the lemma by induction on $\alpha$. The case $\alpha=0$ is obvious.
Set $Q=\delta_{g}^{\alpha}(h,\eta,P)$ and assume that $x\in U\cap O_{g}%
(h,\eta,Q)$. For any open neighborhood $V$ of $x$, there exists $y\in U\cap
V\cap Q$ such that
\[
\eta\leq|h(y)-h(x)|(|g(y)|\vee|g(x)|)\leq M|h(y)-h(x)|.
\]
By the inductive hypothesis, $U\cap Q\subseteq\delta^{\alpha}(h,\frac{\eta}%
{M},P)$. Thus $y\in V\cap\delta^{\alpha}(h,\frac{\eta}{M},P)$. It follows that
$x\in\delta^{\alpha+1}(h,\frac{\eta}{M},P).$ This shows that $U\cap
O_{g}(h,\eta,Q)\subseteq\delta^{\alpha+1}(h,\frac{\eta}{M},P)$ and thus, by
taking closures, that the lemma holds for $\alpha+1$.

Finally, let $\alpha$ be a limit ordinal and assume that the lemma holds for
all $\alpha^{\prime}< \alpha$. Then
\begin{align*}
U \cap\delta^{\alpha}_{g}(h,\eta,P)  &  = U \cap(\cap_{\alpha^{\prime}<
\alpha}\delta^{\alpha^{\prime}}_{g}(h,\eta,P))\\
&  = \cap_{\alpha^{\prime}< \alpha}(U \cap\delta^{\alpha^{\prime}}_{g}%
(h,\eta,P))\\
&  \subseteq\cap_{\alpha^{\prime}< \alpha}\delta^{\alpha^{\prime}}%
(h,\frac{\eta}{M},P) = \delta^{\alpha}(h,\frac{\eta}{M},P).
\end{align*}
This completes the induction.
\end{proof}

Next, we introduce another derivation that is partly responsible for
controlling the second half of $\delta^{\omega^{\alpha}}(gh, \varepsilon,P)$.
Let $g$ be a real valued function on $X$. For any closed set $P$ in $X$, let
$d_{\infty}(g,P)$ be the set of all points $x\in P$ so that for any open
neighborhood $U$ of $x$ and any $n\in{\mathbb{N}}$, there exists $y \in U \cap
P$ such that $|g(y)| > n$. Once again, $d_{\infty}(g,\cdot)$ is a regular derivation.

\begin{lem}
\label{l6} Let $g,h$ be real-valued functions on $X$ and let $\eta> 0$. Let
$P$ be a closed set in $X$. Suppose that $U$ is an open set in $X$ such that
$U \cap(\cup_{a > 0}\delta^{\alpha}(h,a,P)) = \emptyset$ for some
$\alpha<\omega_{1}$. Then $U \cap\delta^{\alpha}_{g}(h,\eta,P)\subseteq
d_{\infty}(g,P)$.
\end{lem}

\begin{proof}
Let $x\in U\backslash d_{\infty}(g,P)$. There exists an open neighborhood $V$
of $x$ so that $V\subseteq U$ and $M=\sup_{x\in V\cap P}|g(x)|<\infty$. By
Lemma \ref{l4'}, $V\cap\delta_{g}^{\alpha}(h,\eta,P)\subseteq\delta^{\alpha
}(h,\frac{\eta}{M},P)$. Since $V\subseteq U$ and $U\cap\delta^{\alpha}%
(h,\frac{\eta}{M},P)=\emptyset$, $V\cap\delta_{g}^{\alpha}(h,\eta
,P)=\emptyset$. In particular, $x\notin\delta_{g}^{\alpha}(h,\eta,P)$.
\end{proof}

\begin{lem}
\label{l7} Let $g,h$ be real-valued functions on $X$ and let $\eta> 0$. Let
$P$ be a closed set in $X$. Suppose that $U$ is an open set in $X$ such that
$U \cap(\cup_{a > 0}\delta^{\alpha}(h,a,P)) = \emptyset$ for some
$\alpha<\omega_{1}$. Then $U \cap\delta^{\alpha\kappa}_{g}(h,\eta,P)\subseteq
d_{\infty}^{\kappa}(g,P)$ for any $\kappa<\omega_{1}$.
\end{lem}

\begin{proof}
The case $\kappa=1$ holds by Lemma \ref{l6}. Assume that the lemma holds for
some $\kappa<\omega_{1}$. Suppose that $x\in U\cap\delta_{g}^{\alpha
(\kappa+1)}(h,\eta,P)$. Let $Q=\delta_{g}^{\alpha\kappa}(h,\eta,P)$. Then
$U\cap(\cup_{a>0}\delta^{\alpha}(h,a,Q))=\emptyset$. By Lemma \ref{l6},
\[
x\in U\cap\delta_{g}^{\alpha(\kappa+1)}(h,\eta,P)=U\cap\delta_{g}^{\alpha
}(h,\eta,Q)\subseteq d_{\infty}(g,Q).
\]
For any open neighborhood $V$ of $x$ and any $n\in{\mathbb{N}}$, there exists
$v\in V\cap U\cap Q$ such that $|g(v)|>n$. By the inductive hypothesis, $U\cap
Q\subseteq d_{\infty}^{\kappa}(g,P)$. Then $v\in V\cap d_{\infty}^{\kappa
}(g,P)$ and $|g(v)|>n$. This proves that $x\in d_{\infty}^{\kappa+1}(g,P)$
\end{proof}

The final derivation that we will require is the following. For each countable
ordinal $\xi$, fix a sequence $(\xi_{n})$ that strictly increases to
$\omega^{\omega^{\xi}}$. Given a real valued function $h$, a countable ordinal
$\xi$ and a closed set $P$ in $X$, define $\Omega_{h,\xi}(P)=\cap_{n}%
[\cup_{\eta>0}\delta^{{\xi_{n}}}(h,\eta,P)]^{\prime}$. Again, $\Omega_{h,\xi}$
is a regular derivation. We are now ready to take control of the part
$\delta_{g}^{\omega^{\alpha}}(h,\frac{\varepsilon}{2},P)$.

\begin{prop}
\label{p7} Let $g,h$ be real-valued functions on $X$ and let $\eta> 0$. For
any $\xi< \omega_{1}$ and any closed set $P$ in $X$,
\[
\delta^{\omega^{\omega^{\xi}}}_{g}(h,\eta,P) \subseteq\Omega_{h,\xi}(P)\cup
d_{\infty}^{\omega^{\omega^{\xi}}}(g,P).
\]

\end{prop}

\begin{proof}
Suppose that $x\in\delta^{\omega^{\omega^{\xi}}}_{g}(h,\eta,P)$ and that
$x\notin\Omega_{h,\xi}(P)$. There exists an open neighborhood $U$ of $x$ and
$n\in{\mathbb{N}}$ such that $U \cap(\cup_{a > 0}\delta^{{\xi_{n}}}(h,a,P)) =
\emptyset$. Taking $\alpha= \xi_{n}$ and $\kappa= \omega^{\omega^{\xi}}$ in
Lemma \ref{l7}, we have
\[
U \cap\delta^{\omega^{\omega^{\xi}}}_{g}(h,\eta,P) = U \cap\delta^{{\xi_{n}%
}\cdot\omega^{\omega^{\xi}}}_{g}(h,\eta,P)\subseteq d_{\infty}^{\omega
^{\omega^{\xi}}}(g,P).
\]
In particular, $x\in d_{\infty}^{\omega^{\omega^{\xi}}}(g,P).$
\end{proof}

Since any iterate of a regular derivation is regular, by Propositions \ref{p0}
and \ref{p7}, we have

\begin{cor}
\label{c8} Let $g,h$ be real-valued functions on $X$ and let $\eta> 0$. For
any $\alpha, \xi< \omega_{1}$ and any closed set $P$ in $X$,
\[
\delta^{\omega^{\omega^{\xi}}\cdot\omega^{\alpha}}_{g}(h,\eta,P)
\subseteq\Omega_{h,\xi}^{\omega^{\alpha}}(P)\cup d_{\infty}^{\omega
^{\omega^{\xi}}\cdot\omega^{\alpha}}(g,P).
\]

\end{cor}

For the remaining results, we adopt the following notation. Let $\kappa$ be a
nonzero countable ordinal. Then $\kappa$ has a standard representation
\[
\kappa=\omega^{\kappa_{1}}m_{1}+\cdots+\omega^{\kappa_{k}}m_{k},
\]
where $\kappa_{1}>\cdots>\kappa_{k}$ are countable ordinals and $m_{1}%
,\dots,m_{k}\in{\mathbb{N}}$. Define the pair $(\xi,\alpha)$ by $\xi
=\kappa_{1}$ and $\alpha=\omega^{\kappa_{1}}(m_{1}-1)+\omega^{\kappa_{2}}%
m_{2}+ \cdots+\omega^{\kappa_{k}}m_{k}$. If $\kappa$ is related to the pair
$(\xi,\alpha)$ as above, we write $\kappa\sim(\xi,\alpha)$. Note that in this
case $\kappa= \omega^{\xi}+\alpha$.

\begin{thm}
\label{t9} Let $h$ be a real-valued function on $X$ and let $\kappa\ $be a
nonzero countable ordinal such that $\kappa\symbol{126}\left(  \xi
,\alpha\right)  $. Suppose that $\cap_{M}h_{M}^{\omega^{\kappa}}(X)=\emptyset$
and that $\Omega_{h,\xi}^{\omega^{\alpha}}(X)=\emptyset$. Then $h\in
{\mathcal{M}}(\kappa,\kappa)$.
\end{thm}

\begin{proof}
Let $g$ be a real-valued function on $X$ so that $\beta(g)\leq\omega^{\kappa}%
$. Fix $\varepsilon>0$. By Proposition \ref{p2'}, for any $M<\infty$,
$\delta_{h}^{\omega^{\kappa}}(g,\varepsilon,X)\subseteq h_{M}^{\omega^{\kappa
}}(X)$. Hence $\delta_{h}^{\omega^{\kappa}}(g,\varepsilon,X)=\emptyset$. Next,
from $\beta(g)\leq\omega^{\kappa}$, we see that $d_{\infty}^{\omega^{\kappa}%
}(g,X)=\emptyset$. Since $\kappa=\omega^{\xi}+\alpha$, by Corollary \ref{c8},
\[
\delta_{g}^{\omega^{\kappa}}(h,\varepsilon,X)=\delta_{g}^{\omega^{\omega^{\xi
}}\cdot\omega^{\alpha}}(h,\varepsilon,X)=\emptyset.
\]
Thus $\delta^{\omega^{\kappa}}(gh,2\varepsilon,X)=\emptyset$ by Proposition
\ref{p1}. Since $\varepsilon>0$ is arbitrary, it follows that $\beta
(gh)\leq\omega^{\kappa}$.
\end{proof}

In the remainder of the section, we prove the converse to Theorem \ref{t9}.
Let $X$ be a Polish space. Fix a metric $d$ that generates the topology on
$X$. Denote by $B(x,\varepsilon)$ the open ball, with respect to $d$, of
radius $\varepsilon$ centered at a point $x\in X$. Recall that the set of
accumulation points of a set $A$ is denoted by $A^{\prime}$. For any countable
ordinal $\kappa$, the $\kappa$-th iteration of this operation is denoted by
$A^{(\kappa)}$.

\begin{lem}
\label{l4.13} Let $x\in h_{M}^{\kappa}( P) $ for some closed set $P$ and let
$M$ be a nonnegative real number. For any $\varepsilon>0,$ there exists a
compact set $F\subseteq B( x,\varepsilon) $ such that $F$ is homeomorphic to
$[ 1,\omega^{\kappa}] $, $\vert h( u) \vert>M$ for all $u\in F\backslash\{ x\}
$ and that $h_{M}^{\kappa}( F) =\{ x\} = F^{(\kappa)} .$
\end{lem}

\begin{proof}
The case $\kappa=0$ holds trivially if we take $F=\{x\}$.

Assume that the lemma holds for some countable ordinal $\kappa$. If $x\in
h_{M}^{\kappa+1}( P) ,$ then there exists a sequence $( x_{k}) \subseteq
h_{M}^{\kappa}( P) \cap B( x,\frac{\varepsilon}{2}) $ converging to $x$ such
that $\vert h( x_{k}) \vert>M$ and $d_{k+1}<\frac{d_{k}}{4}$ for all $k$,
where $d_{k}=d( x_{k},x) $. For each $k$, there exists a compact
$F_{k}\subseteq B( x_{k},\frac{d_{k}}{2}) $, homeomorphic to $[ 1,\omega
^{\kappa}] ,$ such that $\vert h( u) \vert>M$ for all $u\in F_{k}\backslash\{
x_{k}\} $ and $h_{M}^{\kappa}( F_{k}) = \{ x_{k}\} = F_{k}^{(\kappa)} .$
Clearly $F_{k}\cap F_{l}=\emptyset$ if $k\neq l$ and $\overline{\cup F_{k}}=(
\cup F_{k}) \cup\{ x\} \subseteq B(x,\varepsilon) .$ Thus $F=( \cup F_{k})
\cup\{ x\} $ is a compact subset of $B(x,\varepsilon)$ that is homeomorphic to
$[ 1,\omega^{\kappa+1}] .$ Also, $\vert h( u) \vert>M$ for all $u\in
F\backslash\{ x\} $. As each $F_{k}$ is relatively clopen in $F$,
$h_{M}^{\kappa}( F) \cap F_{k}=h_{M}^{\kappa}(F_{k})= \{x_{k}\}$ for all $k$.
Since $h_{M}^{\kappa}(F)$ is closed, it follows that $h_{M}^{\kappa}( F)
=\cup_{k}\{ x_{k}\} \cup\{ x\} $. Hence $h_{M}^{\kappa+1}( F) =\{ x\} .$
Similarly, $F^{(\kappa+1)} = \{x\}$.

Suppose that $\kappa$ is a limit ordinal and that the lemma holds for all
$\kappa^{\prime}<\kappa$. Choose a sequence of ordinals $(\kappa_{k})$ that
strictly increases to $\kappa$. If $x\in h_{M}^{\kappa}( P) $, then $x\in
\cap_{k}h_{M}^{\kappa_{k}+1}( P) $. There exists a sequence $( x_{k}) $ in
$P\cap B(x,\frac{\varepsilon}{2})$ such that $x_{k}\in h_{M}^{\kappa_{k}}( P)
,$ $\vert h( x_{k}) \vert>M$ and $d_{k+1}<\frac{1}{4}d_{k}$, where
$d_{k}=d(x_{k},x)$. By the inductive hypothesis, there are compact sets
$F_{k}\subseteq B( x_{k},\frac{d_{k}}{2}) ,$ homeomorphic to $[ 1,\omega
^{\kappa_{k}}] ,$ such that $\vert h( u) \vert>M$ for all $u\in F_{k}%
\backslash\{ x_{k}\} $ and that $h_{M}^{\kappa_{k}}( F_{k})= \{ x_{k}\} =
F_{k}^{(\kappa_{k})}.$ It is easy to see that $F=( \cup F_{k}) \cup\{ x\} $ is
a compact set in $B(x,\varepsilon)$ that is homeomorphic to $[ 1,\omega
^{\kappa}] .$ By choice, $\vert h( u) \vert>M$ for all $u\in F\backslash\{ x\}
.$ For any $j>k$,
\[
h_{M}^{\kappa_{j}}( F) \cap F_{k}=h_{M}^{\kappa_{j}}( F_{k}) =\emptyset.
\]
Thus $h_{M}^{\kappa_{j}}( F) \subseteq\cup_{k\geq j}F_{k}\cup\{x\}$. Hence
\[
h_{M}^{\kappa}(F)=\cap_{j}h_{M}^{\kappa_{j}}( F) \subseteq\cap_{j}(\cup_{k\geq
j}F_{k}\cup\{x\})=\{x\}.
\]
On the other hand, since $x_{k}\in h_{M}^{\kappa_{k}}(F_{k})\subseteq
h_{M}^{\kappa_{k}}(F)$ for all $k$, and the latter set is closed, $x\in
h_{M}^{\kappa_{k}}(F)$ for all $k$. Thus $x\in h_{M}^{\kappa}(F)$. This shows
that $h_{M}^{\kappa}(F) = \{x\}$. Similarly, $F^{(\kappa)} = \{x\}$.
\end{proof}

The next proposition gives one half of the converse to Theorem \ref{t9}.

\begin{prop}
\label{p9.5}Let $h$ be a real-valued function on $X$ and let $\kappa$ be a
nonzero countable ordinal. Suppose that $\cap_{M}h_{M}^{\omega^{\kappa}%
}(X)\neq\emptyset$. Then $h\notin{\mathcal{M}}(\kappa,\kappa)$.
\end{prop}

\begin{proof}
Suppose that $x\in\cap_{M}h_{M}^{\omega^{\kappa}}(X)$. Let $(\kappa_{n})$ be a
sequence of ordinals strictly increasing to $\omega^{\kappa}.$ For for all
$n\in\mathbb{N},$ $x\in h_{n}^{\omega^{\kappa}}(X)\subseteq h_{n}^{\kappa
_{n}+1}(X).$ Pick $x_{1}\in h_{1}^{\kappa_{1}}(X)\cap B(x,1)$, distinct from
$x$, such that $|h(x_{1})|>1.$ Inductively, for any $n>1$, let $d_{n-1}%
=d(x_{n-1},x)$ and choose $x_{n}\in h_{n}^{\kappa_{n}}(X)\cap B(x,\frac{1}%
{4}d_{n-1})$, distinct from $x$, such that $|h(x_{n})|>n.$ By Lemma
\ref{l4.13}, for each $n,$ there exists a compact set $F_{n}\subseteq
B(x_{n},\frac{1}{2}d_{n})$, homeomorphic to $[1,\omega^{\kappa_{n}}]$ such
that $|h(u)|>n$ for all $u\in F_{n}\setminus\{x_{n}\}$ and that $h_{n}%
^{\kappa_{n}}(F_{n})=\{x_{n}\}=F_{n}^{(\kappa_{n})}.$ Since $|h(x_{n})|>n$ by
choice, we have in fact that $|h(u)|>n$ for all $u\in F_{n}$. Since $F_{n}$ is
homeomorphic to $[1,\omega^{\kappa_{n}}]$, there is a continuous
$\{0,1\}$-valued function $g_{n}$ on $F_{n}$ such that $g_{n}(x_{n})=1$ and
that $D^{\kappa_{n}}(g_{n},1,F_{n})=\{x_{n}\}$. Define $g:X\rightarrow
\mathbb{R}$ by
\[
g(t)=%
\begin{cases}
\frac{g_{n}(t)}{h(t)} & \text{if $t\in F_{n}$ for some $n$},\\
0 & \text{if $t\notin\cup F_{n}$}.
\end{cases}
\]
For all $N\in\mathbb{N},$ $D^{1}(g,\frac{1}{N},X)\subseteq\overline{\cup
F_{n}}=(\cup F_{n})\cup\{x\}.$ If $n>2N$, $|g|<\frac{1}{2N}$ on the set
$F_{n}$. Thus $D^{2}(g,\frac{1}{N},X)\cap F_{n}=\emptyset$. Hence
$D^{2}(g,\frac{1}{N},X)\subseteq\cup_{n=1}^{2N}F_{n}\cup\{x\}$. As $x$ is an
isolated point in $\cup_{n=1}^{2N}F_{n}\cup\{x\}$ and $F_{n}^{(\kappa_{2N}%
+1)}=\emptyset$ if $n\leq2N$, $\beta(g,\frac{1}{N})\leq2+\kappa_{2N}+1$. It
follows that $\beta(g)\leq\omega^{\kappa}.$ However,
\[
(gh)(t)=%
\begin{cases}
g_{n}(t) & \text{if $t\in F_{n}$ for some $n$},\\
0 & \text{if $t\notin\cup F_{n}$}.
\end{cases}
\]
Therefore, $x_{n}\in D^{\kappa_{n}}(g_{n},F_{n},1)\subseteq D^{\kappa_{n}%
}(gh,1,X)$ for all $n$. Since $(x_{n})$ converges to $x$ and $(\kappa_{n})$
increases to $\omega^{\kappa}$, $x\in D^{\omega^{\kappa}}(gh,1,X)$. Therefore,
$\beta(gh)>\omega^{\kappa}$. It follows that $h\notin{\mathcal{M}}%
(\kappa,\kappa)$.
\end{proof}

Let $h$ be a real valued function on $X$ and let $\xi$ and $\zeta$ be
countable ordinals. Say that a closed subset $P$ of $X$ has \emph{property
$(\xi,\zeta)$ with respect to $h$} if for any $x\in P$ and any $\varepsilon
>0$, there are a closed set $Q$ and a function $g:Q\to{\mathbb{N}}$ such that
$Q \subseteq B(x,\varepsilon)$, $g =1$ on $D^{\xi}(g,1,Q)\cup\{x\}$ and $x\in
D^{\zeta}(gh,1,Q)$. A sequence $(x_{n})$ is said to \emph{converge
nontrivially} to $x$ if $(x_{n})$ is a sequence of distinct points, all
different from $x$, that converges to $x$.

\begin{prop}
\label{p10} Let $\alpha, \xi,\zeta$ be countable ordinals. If $P$ is a closed
set in $X$ that has property $(\xi,\zeta)$ with respect to $h$, then
$D^{\alpha}(h,1,P)$ has property $(\xi,\zeta+\alpha)$ with respect to $h$.
\end{prop}

\begin{proof}
We prove the proposition by induction on $\alpha$. The case $\alpha=0$ is the
hypothesis. Assume that the proposition holds for some $\alpha<\omega_{1}$.
Let $x_{\infty}\in D^{\alpha+1}(h,1,P)$. Set $P_{0}=D^{\alpha}(h,1,P)$. Let
$\varepsilon>0$ be given. There exists a sequence $(x_{n})$ in $P_{0}\cap
B(x_{\infty},\varepsilon)$ that converges nontrivially to $x_{\infty}$ and
that for any $N\in{\mathbb{N}}$, there are $m,n\in\mathbb{N}\cup\{\infty\}$,
$m,n\geq N$, so that $|h(x_{m})-h(x_{n})|\geq1$. By the inductive hypothesis,
$P_{0}$ has property $(\xi,\zeta+\alpha)$ with respect to $h$. We can find
pairwise disjoint closed sets $Q_{n}$ in $B(x_{\infty},\varepsilon)$ and
functions $g_{n}:Q_{n}\rightarrow{\mathbb{N}}$ such that $x_{\infty}\notin
Q_{n}$, $\operatorname{diam}Q_{n}\rightarrow0$, $g_{n}=1$ on $D^{\xi}%
(g_{n},1,Q_{n})\cup\{x_{n}\}$ and $x_{n}\in D^{\zeta+\alpha}(g_{n}h,1,Q_{n})$.
Let $Q=\{x_{\infty}\}\cup(\bigcup Q_{n})$. Then $Q$ is a closed set in
$B(x_{\infty},\varepsilon)$. Define $g:Q\rightarrow{\mathbb{N}}$ by
$g(u)=g_{n}(u)$ if $u\in Q_{n}$ and $g(x_{\infty})=1$. We have
\[
D^{\xi}(g,1,Q)\subseteq\{x_{\infty}\}\cup(\bigcup D^{\xi}(g_{n},1,Q_{n})).
\]
Hence $g=1$ on $D^{\xi}(g,1,Q)\cup\{x_{\infty}\}$. Also,
\[
x_{n}\in D^{\zeta+\alpha}(g_{n}h,1,Q_{n})\subseteq D^{\zeta+\alpha
}(gh,1,Q)\text{ for all $n$}.
\]
Since that latter is a closed set, $x_{\infty}\in D^{\zeta+\alpha}(gh,1,Q)$ as
well. Let $U$ be an open neighborhood of $x_{\infty}$. There exists $N$ such
that $x_{n}\in U$ for all $n\geq N$. By choice, there are $m,n\in
\mathbb{N}\cup\{\infty\}$, $m,n\geq N$, such that
\[
|(gh)(x_{m})-(gh)(x_{n})|=|h(x_{m})-h(x_{n})|\geq1.
\]
Since $x_{m},x_{n}\in D^{\zeta+\alpha}(gh,1,Q)$, $x_{\infty}\in D^{\zeta
+\alpha+1}(gh,1,Q)$. This shows that $D^{\alpha+1}(h,1,P)$ has property
$(\xi,\zeta+\alpha+1)$ with respect to $h$.

Now suppose that $\alpha$ is a countable limit ordinal and that the
proposition has been verified for all $\alpha^{\prime}<\alpha$. Let
$x_{\infty}\in D^{\alpha}(h,1,P)$ and $\varepsilon>0$ be given. Choose a
sequence $(\alpha_{n})$ of ordinals strictly increasing to $\alpha$. There is
a sequence $(x_{n})$ converging nontrivially to $x_{\infty}$ so that $x_{n}\in
D^{\alpha_{n}}(h,1,P)\cap B(x_{\infty},\varepsilon)$ for all $n$. By the
inductive hypothesis, $D^{\alpha_{n}}(h,1,P)$ has property $(\xi,\zeta
+\alpha_{n})$ with respect to $h$. Let $\varepsilon>0$ be given. We can find
pairwise disjoint closed sets $Q_{n}$ in $B(x_{\infty},\varepsilon)$ and
functions $g_{n}:Q_{n}\rightarrow{\mathbb{N}}$ such that $x_{\infty}\notin
Q_{n}$, $\operatorname{diam}Q_{n}\rightarrow0$, $g_{n}=1$ on $D^{\xi}%
(g_{n},1,Q_{n})\cup\{x_{n}\}$ and $x_{n}\in D^{\zeta+\alpha_{n}}%
(g_{n}h,1,Q_{n})$. Let $Q=\{x_{\infty}\}\cup(\bigcup Q_{n})$. Then $Q$ is a
closed set in $B(x_{\infty},\varepsilon)$. Define $g:Q\rightarrow{\mathbb{N}}$
by $g(u)=g_{n}(u)$ if $u\in Q_{n}$ and $g(x_{\infty})=1$. For each $n$,
$D^{\xi}(g,1,Q_{n})=D^{\xi}(g_{n},1,Q_{n})$. Thus $D^{\xi}(g,1,Q)\subseteq
(\bigcup D^{\xi}(g_{n},1,Q_{n}))\cup\{x_{\infty}\}$. Hence $g=1$ on $D^{\xi
}(g,1,Q)\cup\{x_{\infty}\}$. Finally, for each $n$,
\[
x_{n}\in D^{\zeta+\alpha_{n}}(g_{n}h,1,Q_{n})\subseteq D^{\zeta+\alpha_{n}%
}(gh,1,Q).
\]
Since the latter set is closed, $x_{\infty}\in D^{\zeta+\alpha_{n}}(gh,1,Q)$
for all $n$. Thus $x_{\infty}\in D^{\zeta+\alpha}(gh,1,Q)$. This proves that
$D^{\alpha}(h,1,P)$ has property $(\xi,\zeta+\alpha)$ with respect to $h$.
\end{proof}

Observe that for any nonzero $c\in{\mathbb{R}}$, any real-valued function $h$
on $X$, any ordinals $\alpha,\xi$ and any closed set $P$ in $X$,
$\Omega^{\alpha}_{ch,\xi}(P) = \Omega^{\alpha}_{h,\xi}(P)$. Also, by
Proposition \ref{p4.3}, $\Omega^{\alpha}_{h,\xi}(P) =\cap_{n}[\cup_{\eta
>0}D^{\xi_{n}}(h,\eta,P)]^{\prime}$.

\begin{prop}
\label{p11} Let $\alpha, \xi$ be countable ordinals. For any closed set $P$ in
$X$, the set $\Omega^{\alpha}_{h,\xi}(P)$ has property $(\alpha,\omega
^{\omega^{\xi}}\cdot\alpha)$ with respect to $h$.
\end{prop}

\begin{proof}
We prove the proposition by induction on $\alpha$. The case $\alpha=0$ is
trivial. Assume that the proposition holds for some $\alpha<\omega_{1}$. Let
$x\in\Omega_{h,\xi}^{\alpha+1}(P)$ and let $\varepsilon>0$ be given. Set
$P_{0}=\Omega_{h,\xi}^{\alpha}(P)$. For any open neighborhood $U$ of $x$,
there exists $k_{n}\in{\mathbb{N}}$ such that $U\cap D^{\xi_{n+1}}(h,\frac
{1}{k_{n}},P_{0})\neq\emptyset$. Thus there is a sequence $(x_{n})$ converging
nontrivially to $x$ so that $x_{n}\in D^{\xi_{n}}(h,\frac{1}{k_{n}}P_{0},)$.
We may assume that $x_{n}\in B(x,\varepsilon)$ for all $n$. By the inductive
hypothesis, $P_{0}=\Omega_{k_{n}h,\xi}^{\alpha}(P)$ has property
$(\alpha,\omega^{\omega^{\xi}}\cdot\alpha)$ with respect to $k_{n}h$. By
Proposition \ref{p10}, $D^{\xi_{n}}({h},\frac{1}{k_{n}},P_{0})=D^{\xi_{n}%
}(k_{n}h,1,P_{0})$ has property $(\alpha,\omega^{\omega^{\xi}}\cdot\alpha
+\xi_{n})$ with respect to $k_{n}h$. Thus, we can find pairwise disjoint
closed sets $Q_{n}$ in $B(x,\varepsilon)$ and functions $g_{n}:Q_{n}%
\rightarrow{\mathbb{N}}$ such that $x\notin Q_{n}$, $\operatorname{diam}%
Q_{n}\rightarrow0$, $g_{n}=1$ on $D^{\alpha}(g_{n},1,Q_{n})\cup\{x_{n}\}$ and
$x_{n}\in D^{\omega^{\omega^{\xi}}\cdot\alpha+\xi_{n}}(k_{n}g_{n}h,1,Q_{n})$.
Let $Q=\{x\}\cup(\bigcup Q_{n})$. Then $Q$ is a closed set in $B(x,\varepsilon
)$. Define $g:Q\rightarrow{\mathbb{N}}$ by setting $g(u)=k_{n}g_{n}(u)$ if
$u\in Q_{n}$ and $g(x)=1$. Then
\[
D^{\alpha}(g,1,Q_{n})=D^{\alpha}(g_{n},\frac{1}{k_{n}},Q_{n})=D^{\alpha}%
(g_{n},1,Q_{n}),
\]
where the second equality holds since $g_{n}$ is integer valued. This shows
that $g=k_{n}$ on the set $D^{\alpha}(g,1,Q_{n})$. Since $D^{\alpha
}(g,1,Q)\subseteq\cup_{n}D^{\alpha}(g,1,Q_{n})\cup\{x\}$ and each $Q_{n}$ is
relatively clopen in $Q$, it follows that $D^{\alpha+1}(g,1,Q)\subseteq\{x\}$.
Consequently, $g=1$ on $D^{\alpha+1}(g,1,Q)\cup\{x\}$. By choice of $Q_{n}$
and $g_{n}$, we also have $x_{n}\in D^{\omega^{\omega^{\xi}}\cdot\alpha
+\xi_{n}}(gh,1,Q_{n})\subseteq D^{\omega^{\omega^{\xi}}\cdot\alpha+\xi_{n}%
}(gh,1,Q)$ for all $n$. Hence, for all $n\geq m$, $x_{n}\in D^{\omega
^{\omega^{\xi}}\cdot\alpha+\xi_{m}}(gh,1,Q)$. As the latter is a closed set,
$x\in D^{\omega^{\omega^{\xi}}\cdot\alpha+\xi_{m}}(gh,1,Q)$. Since $(\xi_{m})$
increases to $\omega^{\omega^{\xi}}$, it follows that $x\in D^{\omega
^{\omega^{\xi}}\cdot\alpha+\omega^{\omega^{\xi}}}(gh,1,Q)=D^{\omega
^{\omega^{\xi}}\cdot(\alpha+1)}(gh,1,Q)$. This proves that $\Omega_{h,\xi
}^{\alpha+1}(P)$ has property $(\alpha+1,\omega^{\omega^{\xi}}\cdot
(\alpha+1))$ with respect to $h$.

Now assume that $\alpha$ is a limit ordinal and that the proposition has been
proved for all $\alpha^{\prime}< \alpha$. Choose a sequence $(\alpha_{n})$ of
ordinals strictly increasing to $\alpha$. Suppose that $x\in\Omega^{\alpha
}_{h,\xi}(P)$ and let $\varepsilon> 0$ be given. Then $x\in\Omega^{\alpha
_{n}+1}_{h,\xi}(P)$ for all $n$. There is a sequence $(x_{n})$ that converges
nontrivially to $x$ and that $x_{n} \in\Omega^{\alpha_{n}}_{h,\xi}(P) \cap
B(x,\varepsilon)$ for all $n$. By the inductive hypothesis, $\Omega
^{\alpha_{n}}_{h,\xi}(P)$ has property $(\alpha_{n},\omega^{\omega^{\xi}}%
\cdot\alpha_{n})$ with respect to $h$. Thus one can find pairwise disjoint
closed sets $Q_{n}$ in $B(x,\varepsilon)$ and functions $g_{n}:Q_{n}%
\to{\mathbb{N}}$ such that $x\notin Q_{n}$, $\operatorname{diam} Q_{n} \to0$,
$g_{n} =1$ on $D^{\alpha_{n}}(g_{n},1,Q_{n})\cup\{x_{n}\}$ and that $x_{n} \in
D^{\omega^{\omega^{\xi}}\cdot\alpha_{n}}(g_{n}h,1,Q_{n})$ for all $n$. Let $Q
= \{x\}\cup(\bigcup Q_{n})$. Then $Q$ is a closed set in $B(x,\varepsilon)$.
Define $g:Q\to{\mathbb{N}}$ by setting $g(u) = g_{n}(u)$ if $u \in Q_{n}$ and
$g(x) = 1$. For each $n$, $D^{\alpha_{n}}(g,1,Q_{n}) = D^{\alpha_{n}}%
(g_{n},1,Q_{n})$. Thus $D^{\alpha}(g,1,Q) \subseteq\{x\} \cup(\bigcup
D^{\alpha_{n}}(g_{n},1,Q_{n}))$. It follows that $g =1$ on $D^{\alpha
}(g,1,Q)\cup\{x\}$. For $n\geq m$, we also have
\[
x_{n} \in D^{\omega^{\omega^{\xi}}\cdot\alpha_{n}}(g_{n}h,1,Q_{n}) \subseteq
D^{\omega^{\omega^{\xi}}\cdot\alpha_{m}}(gh,1,Q).
\]
Since that last set is closed, $x \in D^{\omega^{\omega^{\xi}}\cdot\alpha_{m}%
}(gh,1,Q)$ for all $m$. Therefore, $x \in D^{\omega^{\omega^{\xi}}\cdot\alpha
}(gh,1,Q)$. This proves that $\Omega^{\alpha}_{h,\xi}(P)$ has property
$(\alpha,\omega^{\omega^{\xi}}\cdot\alpha)$ with respect to $h$. The induction
is complete.
\end{proof}

Recall that the relation $\kappa\sim(\xi,\alpha)$ was defined prior to Theorem
\ref{t9}.

\begin{prop}
\label{p13} Let $h$ be a real-valued function on $X$ and let $\kappa$ be a
countable ordinal, where $\kappa\sim(\xi,\alpha)$. Suppose that $\Omega
^{\omega^{\alpha}}_{h,\xi}(X) \neq\emptyset$. Then $h \notin{\mathcal{M}%
}(\alpha+1, \kappa)$.
\end{prop}

\begin{proof}
By Proposition \ref{p11}, $\Omega_{h,\xi}^{\omega^{\alpha}}(X)$ has property
$(\omega^{\alpha},\omega^{\omega^{\xi}}\cdot\omega^{\alpha})=(\omega^{\alpha
},\omega^{\kappa})$ with respect to $h$. Let $x\in\Omega_{h,\xi}%
^{\omega^{\alpha}}(X)$. There exist a closed set $Q$ and a function
$g:Q\rightarrow{\mathbb{N}}$ such that $g=1$ on $D^{\omega^{\alpha}}(g,1,Q)$
and $x\in D^{\omega^{\kappa}}(gh,1,Q)$. Since $g$ is integer valued,
$D^{\omega^{\alpha}}(g,\varepsilon,Q)\subseteq D^{\omega^{\alpha}}(g,1,Q)$ for
any $\varepsilon>0$. Define $G:X\rightarrow{\mathbb{R}}$ by $G(u)=g(u)$ for
$u\in Q$ and $G(u)=0$ otherwise. Clearly, for any $\varepsilon>0$,
\[
D^{1+\omega^{\alpha}}(G,\varepsilon,X)\subseteq D^{\omega^{\alpha}%
}(G,\varepsilon,Q)=D^{\omega^{\alpha}}(g,\varepsilon,Q)\subseteq
D^{\omega^{\alpha}}(g,1,Q).
\]
Hence $G=g=1$ on $D^{1+\omega^{\alpha}}(G,\varepsilon,X)$. Thus $\beta
(G)\leq1+\omega^{\alpha}+1\leq\omega^{\alpha+1}$. However, $D^{\omega^{\kappa
}}(Gh,1,X)\supseteq D^{\omega^{\kappa}}(Gh,1,Q)=D^{\omega^{\kappa}%
}(gh,1,Q)\neq\emptyset$. Therefore, $\beta(Gh)>\omega^{\kappa}$. This proves
that $h\notin{\mathcal{M}}(\alpha+1,\kappa)$.
\end{proof}

If $0 < \kappa\sim(\xi,\alpha)$, then $\alpha+1\leq\kappa$ and thus
${\mathcal{M}}(\kappa,\kappa) \subseteq{\mathcal{M}}(\alpha+1,\kappa)$. The
next theorem incorporates Theorem \ref{t9} and its converse; it follows easily
from Theorem \ref{t9}, Propositions \ref{p9.5} and \ref{p13}.

\begin{thm}
\label{t14} Let $h$ be a real-valued function on $X$ and let $\kappa$ be a
nonzero countable ordinal. Then $h \in{\mathcal{M}}(\kappa,\kappa)$ if and
only if $\cap_{M}h_{M}^{\omega^{\kappa}}(X)=\emptyset$ and $\Omega
^{\omega^{\alpha}}_{h,\xi}(X) = \emptyset$.
\end{thm}

\end{document}